\documentclass{amsart}

\usepackage[left=25mm,right=25mm,top=25mm,bottom=25mm]{geometry}
\usepackage{graphicx}
\usepackage{hyperref}
\usepackage{mathpazo}
\usepackage{microtype}
\usepackage{multicol}
\usepackage{booktabs}

\usepackage{latexsym}
\usepackage{multirow}
\usepackage{setspace}

\theoremstyle{plain}
\newtheorem{theorem}{Theorem}
\newtheorem{lemma}[theorem]{Lemma}
\newtheorem{proposition}[theorem]{Proposition}

\newtheorem{corollary}[theorem]{Corollary}

\theoremstyle{definition}

\newcommand{\x}{{\mathbf x}}

\newcommand{\bc}{\mathbb{C}}

\newcommand{\bp}{\mathbb{P}}
\newcommand{\bq}{\mathbb{Q}}
\newcommand{\bz}{\mathbb{Z}}

\newcommand{\modm}{\mathcal{M}}

\newcommand{\res}[1]{\begin{array}[d]{l}\\{\rm Res}\\^{#1}\end{array}\hspace{-1mm}}

\allowdisplaybreaks[1]

\begin{document}

\title{Gromov-Witten invariants of $\bp^1$ and Eynard-Orantin invariants.}
\author{Paul Norbury \and Nick Scott}
\address{Department of Mathematics and Statistics, The University of Melbourne, Victoria 3010, Australia}
\email{pnorbury@ms.unimelb.edu.au, N.Scott@ms.unimelb.edu.au}

\subjclass[2000]{14N35; 32G15; 30F30; 05A15}
\date{August 22, 2011}
\begin{abstract}
We prove that stationary Gromov-Witten invariants of $\bp^1$ arise as the Eynard-Orantin invariants of the spectral curve $x=z+1/z$, $y=\ln{z}$.  As an application we show that tautological intersection numbers on the moduli space of curves arise in the asymptotics of large degree Gromov-Witten invariants of $\bp^1$.
\end{abstract}

\maketitle

\tableofcontents

\section{Introduction} \label{introduction}
As a tool for studying enumerative problems in geometry Eynard and Orantin \cite{EOrInv} define invariants of any compact Torelli marked Riemann surface $C$, equipped with two meromorphic functions $x$ and $y$ with the property that the branch points of $x$ are simple and the map
\[ \begin{array}[b]{rcl} C&\to&\bc^2\\p&\mapsto& (x(p),y(p))\end{array}\]
%\begin{align*} C&\to\bc^2\\ p&\mapsto (x(p),y(p)) \end{align*}
is an immersion.  For every $(g,n)\in\bz^2$ with $g\geq 0$ and $n>0$ the Eynard-Orantin invariant $\omega^g_n(p_1,...,p_n)$ for $p_i\in C$ is a multidifferential, i.e. a tensor product of meromorphic 1-forms on the product $C^n$.  One can make sense of $F^g=\omega^g_0$ using a recursion between $\omega^g_{n+1}$ and $\omega^g_n$ known as the dilaton equation.  See Section~\ref{sec:EO} for more details and the definition of the invariants.

Important examples of the Eynard-Orantin invariants, using different choices of $(C,x,y)$, store intersection numbers over the moduli space of curves \cite{EynRec}; simple Hurwitz numbers \cite{BEMSMat,BMaHur,EMSLap}; a count of lattice points in the moduli space of curves \cite{NorStr}; and conjecturally local Gromov-Witten invariants of (non-compact) toric Calabi-Yau 3-folds \cite{BKMPRem,MarOpe} and Chern-Simons invariants of 3-manifolds \cite{DFMVol}.

The Gromov-Witten invariants of $\bp^1$ have been studied and well understood over the last ten years \cite{GetTod,OPaGro,OPaGro2,OPaVir}.  In this paper we show that the Gromov-Witten invariants of $\bp^1$ arise as Eynard-Orantin invariants, and how this approach brings new insight to the Gromov-Witten invariants.  We also hope to gain a better understanding of the 
Eynard-Orantin invariants.  The example in this paper, together with the simple Hurwitz problem \cite{EMSLap} and the count of lattice points in the moduli space of curves which also corresponds to a Hurwitz problem \cite{NorStr}, raises the question: is the relationship of  Eynard-Orantin invariants to Hurwitz problems a more general phenomenon?

Assemble the connected stationary Gromov-Witten invariants
\begin{equation}\label{gropoint}
\left\langle \prod_{i=1}^n \tau_{b_i}(\omega)\right\rangle^g_d=\int_{[\overline{\modm}^g_n(\bp^1,d)]^{vir}}\prod_{i=1}^n \psi_i^{b_i}ev_i^\ast(\omega)
\end{equation}
where $d$ is determined by $\sum_{i=1}^n b_i=2g-2+2d$, into the generating function
\[\Omega^g_n(x_1,...,x_n)=\sum_{\bf b}\left\langle \prod_{i=1}^n \tau_{b_i}(\omega) \right\rangle^g_d\cdot\prod_{i=1}^n(b_i+1)!x_i^{-b_i-2}dx_i\]
which is a multidifferential.  See Section~\ref{sec:GW} for a more detailed definition of  Gromov-Witten invariants.

The Eynard-Orantin invariants $\omega^g_n$ are defined for any genus 0 compact Riemann surface $C$ equipped with two {\em meromorphic} functions $x$ and $y$.   Nevertheless, by taking sequences of meromorphic functions one can extend the definition to allow $y$ to be any analytic function defined on a domain of $C$ containing the branch points of $x$.  In particular, consider
\begin{equation}  \label{eq:lnz}
C=\begin{cases}x=z+1/z\\ 
y=\ln{z}.
\end{cases}
\end{equation} 
The Riemann surface $C$ is defined via the meromorphic function $x(z)$.  The function $y(z)=\ln{z}\sim\sum\frac{(1-z^2)^k}{\hspace{-3mm}-2k}$ is to be understood as the sequence of partial sums $y_N=\displaystyle{\sum_1^N}$$\frac{(1-z^2)^k}{\hspace{-3mm}-2k}$.  Each invariant requires only a finite $y_N$---for fixed $(g,n)$ the sequence of invariants $\omega^g_n$ of $(C,x,y_N)$ stabilises for $N\geq 6g-6+2n$.
%\vspace{-.8cm}
\begin{theorem}   \label{th:main}
For $g=0$ and 1 and $2g-2+n>0$, the Eynard-Orantin invariants of the curve $C$ defined in (\ref{eq:lnz}) agree with the generating function for the Gromov-Witten invariants of $\bp^1$:
\[ \omega^g_n\sim\Omega^g_n(x_1,...,x_n).\]
More precisely, $\Omega^g_n(x_1,...,x_n)$ gives an analytic expansion of $\omega^g_n$ around a branch of $\{x_i=\infty\}$.
\end{theorem}
In the two exceptional cases $(g,n)=(0,1)$ and $(0,2)$, the invariants $\omega^g_n$ are not analytic at $x_i=\infty$.  We can again get analytic expansions around a branch of $\{x_i=\infty\}$ by removing their singularities at $x_i=\infty$ as follows:  
\begin{equation}  \label{eq:excep} 
\omega^0_1+\ln{x_1}dx_1\sim\Omega^0_1(x_1),\quad\omega^0_2-\displaystyle\frac{dx_1dx_2}{(x_1-x_2)^2}\sim\Omega^0_2(x_1,x_2).
\end{equation}

Theorem~\ref{th:main} gives an extremely efficient way to calculate the Gromov-Witten invariants of $\bp^1$.  It also produces a general form of the invariants that reduces to the calculation of a collection of polynomials.
\begin{theorem}\label{th:GWquasi}For $g=0$ and 1, the stationary Gromov-Witten invariants of $\bp^1$ are of the form
\begin{equation}\label{eq:quasi}
\left\langle \prod_{i=1}^k\tau_{2u_i}(\omega)\prod_{i=k+1}^n\tau_{2u_i-1}(\omega)\right\rangle^g=\frac{u_{k+1}\cdots u_n}{\prod_{i=1}^n u_i!^2}p^g_{n,k}(u_1,\dots,u_n)
\end{equation}
where $p^g_{n,k}(u_1,\dots,u_n)$ is a polynomial of degree $3g-3+n$ in the $u_i$'s, symmetric in the first $k$ and the last $n-k$ variables, with top coefficients $c_{\beta}$ of $u_1^{\beta_1}\cdots u_n^{\beta_n}$ given by 
\begin{equation}  \label{eq:gwcoeff}
c_{\beta}=2^g\int_{\overline{\modm}_{g,n}}\psi_1^{\beta_1}...\psi_n^{\beta_n}\end{equation}
for $|\beta|=3g-3+n$. 
\end{theorem}
Again the exceptional cases are $(g,n)=(0,1)$ and $(0,2)$, where we interpret a degree $3g-3+n$ polynomial to mean a rational function given by the reciprocal of a degree 2, respectively degree 1, polynomial.

The asymptotic behaviour of Eynard-Orantin invariants near branch points of $x$ is governed by the local behaviour of the curve $C$ there,  \cite{EOrTop}.  By assumption the local behaviour is described by $x=y^2$ which, as a global curve, has Eynard-Orantin invariants that store tautological intersection numbers over the compactified moduli space of curves $\overline{\modm}_{g,n}$.  This supplies the top coefficients (\ref{eq:gwcoeff}) and enables one to relate the asymptotic behaviour of the Gromov-Witten invariants of $\bp^1$ to tautological intersection numbers over the compactified moduli space of curves $\overline{\modm}_{g,n}$.
\begin{corollary}  \label{th:asym}
For $g=0$ and 1 and $2g-2+n>0$, the stationary Gromov-Witten invariants of $\bp^1$ behave asymptotically as
\begin{equation}\label{eq:asym}
\left\langle \prod_{i=1}^k\tau_{2u_i}(\omega)\prod_{i=k+1}^n\tau_{2u_i-1}(\omega)\right\rangle^g
\sim
\frac{u_{k+1}\cdots u_n}{\prod_{i=1}^n u_i!^2}\sum_{|\beta|=3g-3+n}u_1^{\beta_1}\cdots u_n^{\beta_n}\cdot 2^g\int_{\overline{\modm}_{g,n}}\psi_1^{\beta_1}...\psi_n^{\beta_n}.
\end{equation}
\end{corollary}
In the exceptional cases $(g,n)=(0,1)$ and $(0,2)$, the asymptotic form is given by the exact formulae in Section~\ref{sec:formulae}.

Section~\ref{sec:EO} defines the Eynard-Orantin invariants and  proves recursions for the Eynard-Orantin invariants of the curve (\ref{eq:lnz}) analogous to recursions satisfied by the Gromov-Witten invariants of $\bp^1$.  The definition of Gromov-Witten invariants is contained in Section~\ref{sec:GW}.  Section~\ref{sec:proofofmain} begins by proving a weaker result than Theorem~\ref{th:GWquasi} which is essentially that $\Omega^g_n$ is analytic and extends to a meromorphic multidifferential on a compact Riemann surface, before proving the main results.  Section~\ref{sec:vir} describes the relationship between the defining recursion relations for the Eynard-Orantin invariants and the Virasoro constraints satisfied by the Gromov-Witten invariants of $\bp^1$.  Numerical checks show that the genus constraint in Theorem~\ref{th:main} and hence also in Theorem~\ref{th:GWquasi} and Corollary~\ref{th:asym} should be unnecessary.  Section~\ref{sec:mat} gives a non-rigorous matrix model proof of Theorem~\ref{th:main} that holds for all genus.  Section~\ref{sec:formulae} contains explicit formulae for Eynard-Orantin invariants and Gromov-Witten invariants of $\bp^1$.

{\em Acknowledgements.} The authors would like to thank Norman Do for useful comments.

\section{Eynard-Orantin invariants.}  \label{sec:EO}

For every $(g,n)\in\bz^2$ with $g\geq 0$ and $n>0$ the Eynard-Orantin invariant of a Torelli marked Riemann surface with meromorphic functions $(C,x,y)$ is  a multidifferential $\omega^g_n(p_1,...,p_n)$, i.e. a tensor product of meromorphic 1-forms on the product $C^n$, where $p_i\in C$.  Recall that a {\em Torelli marking} of $C$ is a choice of symplectic basis $\{a_i, b_i\}_{i=1,..,g}$ of the first homology group $H_1(\bar{C})$ of the compact closure $\bar{C}$ of $C$.  In particular, a genus 0 surface $C$ requires no Torelli marking.  When $2g-2+n>0$, $\omega^g_n(p_1,...,p_n)$ is defined recursively in terms of local  information around the poles of $\omega^{g'}_{n'}(p_1,...,p_n)$ for $2g'+2-n'<2g-2+n$.  Equivalently, the $\omega^{g'}_{n'}(p_1,...,p_n)$ are used as kernels on the Riemann surface.  This is a familiar idea, the main example being the Cauchy kernel which gives the derivative of a function in terms of the bidifferential $dwdz/(w-z)^2$ as follows
\[ f'(z)dz=\res{w=z}\frac{f(w)dwdz}{(w-z)^2}=-\sum_{\alpha}\res{w=\alpha}\frac{f(w)dwdz}{(w-z)^2}\]
where the sum is over all poles $\alpha$ of $f(w)$.  

The Cauchy kernel generalises to a bidifferential $B(w,z)$ on any Riemann surface $C$ which arises from the meromorphic differential $\eta_w(z)dz$ unique up to scale which has a double pole at $w\in C$ and all $A$-periods vanishing.   The scale factor can be chosen so that $\eta_w(z)dz$ varies holomorphically in $w$ and transforms as a 1-form in $w$, hence it is naturally expressed as the unique bidifferential on $C$ 
\[ B(w,z)=\eta_w(z)dwdz,\quad \oint_{A_i}B=0,\quad B(w,z)\sim\frac{dwdz}{(w-z)^2} {\rm\  near\ }w=z.\]  
It is symmetric in $w$ and $z$.  We will call $B(w,z)$ the {\em Bergmann Kernel}, following \cite{EOrInv}.  It is called the fundamental normalised differential of the second kind on $C$ in \cite{FayThe}.  Recall that a meromorphic differential is {\em normalised} if its $A$-periods vanish and it is of the {\em second kind} if its residues vanish.  The Bergmann Kernel is used to express a normalised differential of the second kind in terms of local  information around its poles. 

Since each branch point $\alpha$ of $x$ is simple, for any point $p\in C$ close to $\alpha$ there is a unique point $\hat{p}\neq p$ close to $\alpha$ such that $x(\hat{p})=x(p)$.  The recursive definition of $\omega^g_n(p_1,...,p_n)$ uses only local information around branch points of $x$ and makes use of the well-defined map $p\mapsto\hat{p}$ there. The invariants are defined as follows.
\begin{align*}
\omega^0_1&=-ydx(z)\nonumber\\
% \label{eq:berg}
\omega^0_2&=B(z_1,z_2)
\end{align*}
For $2g-2+n>0$,
\begin{equation}  \label{eq:EOrec}
\omega^g_{n+1}(z_0,z_S)=\sum_{\alpha}\hspace{-2mm}\res{z=\alpha}K(z_0,z)\hspace{-.5mm}\biggr[\omega^{g-1}_{n+2}(z,\hat{z},z_S)+\hspace{-5mm}\displaystyle\sum_{\begin{array}{c}_{g_1+g_2=g}\\_{I\sqcup J=S}\end{array}}\hspace{-5mm}
\omega^{g_1}_{|I|+1}(z,z_I)\omega^{g_2}_{|J|+1}(\hat{z},z_J)\biggr]
\end{equation}
where the sum is over branch points $\alpha$ of $x$, $S=\{1,...,n\}$, $I$ and $J$ are non-empty and 
\[\displaystyle K(z_0,z)=\frac{-\int^z_{\hat{z}}B(z_0,z')}{2(y(z)-y(\hat{z}))dx(z)}\] is well-defined in the vicinity of each branch point of $x$.   Note that the quotient of a differential by the differential $dx(z)$ is a meromorphic function.  The recursion (\ref{eq:EOrec}) depends only on the meromorphic differential $ydx$ and the map $p\mapsto\hat{p}$ around branch points of $x$.  For $2g-2+n>0$, each $\omega^g_n$ is a symmetric multidifferential with poles only at the branch points of $x$, of order $6g-4+2n$, and zero residues.

For $2g-2+n>0$, the invariants satisfy the identity 
\[\sum_{x(z)=x}\omega^g_{n+1}(z_S,z)=0\] 
and the string and dilaton equations \cite{EOrInv}:
\begin{align}
\sum_{\alpha}\res{z=\alpha}y(z)x(z)^m\omega^g_{n+1}(z_S,z)&=-\sum_{i=1}^n\partial_{z_i}\left(\frac{x(z_i)^m\omega^g_n(z_S)}{dx(z_i)}\right),\indent ~~~m=0,1\label{EOstring}\\
\sum_{\alpha}\res{z=\alpha}\Phi(z)\omega^g_{n+1}(z_S,z)&=(2g-2+n)\omega^g_n(z_S)\label{EOdilaton}
\end{align}
where the sum is over the branch points $\alpha$ of $x$, $\Phi(z)=\int^z ydx(z')$ is an arbitrary antiderivative and $z_S=(z_1,\dots,z_n)$.

When $y$ is not a meromorphic function on $C$ and is merely analytic in a domain containing the branch points of $x$, we approximate it by a sequence of meromorphic functions $y^{(N)}$ which agree with $y$ at the branch points of $x$ up to the $N$th derivatives.  The sequence $y^{(N)}$ does not necessarily converge to $y$.  For example, the partial sums $y^{(N)}$ of
\[ y(z)=\ln{z}\sim\sum\frac{(1-z^2)^k}{\hspace{-3mm}-2k}\]
give a divergent asymptotic expansion for $\ln(z)$ at $z=0$ in the region $Re(z^2)>0$.  

The meromorphic functions $y^{(N)}$ can be used in the recursions defining $\omega^g_n$ in place of $y(z)$ since they contain the same local information around $z=\pm 1$ up to order $N$.  More precisely, to define $\omega^g_n$ for $(C,x,y)$ it is sufficient to use $(C,x,y^{(N)})$ for any $N\geq 6g-6+2n$.

\subsection{Polynomial behaviour}
In this section we consider the family of curves:
\begin{equation}  \label{generalC}
\tilde{C}=\begin{cases}x=z+1/z\\ 
y=y(z)
\end{cases}
\end{equation}
for $y(z)$ any analytic function defined on a domain of $\bc$ containing $\pm 1$. 

With respect to the local coordinate $x$ on $C$ each invariant $\omega^g_n$ has an analytic expansion around a branch of $x=\infty$.  Define the coefficients of this expansion 
\begin{equation*}
\omega^g_n=:\sum_{b_1,\dots,b_n=0}^\infty \frac{M^g_{n,k}(b_1,\dots,b_n)}{x_1^{b_1+1}\cdots x_n^{b_n+1}}dx_1\cdots dx_n.
\end{equation*}
for $k$ the number of odd $b_i$.  We may abuse this notation by writing $M^g_n=M^g_{n,k}$ when $k$ is clear.  

In \cite{NScPol} it was shown that Eynard-Orantin invariants of such a curve can be expressed via polynomials:
\begin{lemma}[\cite{NScPol}]  \label{zexpand}
For the curve $x=z+1/z$, $y=y(z)$ and $2g-2+n>0$, $\omega_n^g(z_1,...,z_n)$ has an expansion around $\{z_i=0\}$ given by
\begin{equation}  \label{eq:expand}
\omega_n^g(z_1,..,z_n)=\frac{d}{dz_1}\dots\frac{d}{dz_n}\sum_{b_i>0}N^g_n(b_1,\dots,b_n)z_1^{b_1}\dots z_n^{b_n} dz_1\dots dz_n
\end{equation}
where $N^g_n$ is a symmetric quasi-polynomial in the $b_i^2$ of  degree $3g-3+n$, dependent on the parity of the $b_i$.
\end{lemma}
Recall that a function on $\bz^n$ is quasi-polynomial if it is polynomial on each coset of a sublattice $\Gamma \subset \bz^n$ and it is symmetric if it is invariant under the permutation group $S_n$.  In particular, each polynomial is invariant under permutations that preserve the corresponding coset.  The function $N^g_n$ is polynomial on each coset of $2\bz^n \subset \bz^n$.  By symmetry, we can represent its $2^n$ polynomials by the $n$ polynomials $N^g_{n,k}(b_1,...,b_n)$, for $k=1,...,n$, symmetric in $b_1,...,b_k$ and $b_{k+1},...,b_n$ corresponding to the first $k$ variables being odd.
\begin{lemma} [\cite{NScPol}] \label{Nintersection}
The coefficients of the top homogeneous degree terms in the polynomial $N^g_{n,k}(b_1,...,b_n)$, defined above, can be expressed in terms of intersection numbers of $\psi$ classes on $\overline{\modm}_{g,n}$.  For $\sum_i\beta_i=3g-3+n$, the coefficient $v_{\beta}$ of $\prod b_i^{2\beta_i}$ is
\[v_{\beta}=\frac{y'(1)^{2-2g-n}+(-1)^ky'(-1)^{2-2g-n}}{2^{5g-5+2n}\beta_1!...\beta_n!}\int_{\overline{\modm}_{g,n}}\psi_1^{\beta_1}...\psi_n^{\beta_n}.\]
\end{lemma}

In particular, the proofs are constructive, showing how to calculate such polynomials from $\omega_n^g$'s and lead to explicit formulae for the $M^g_n$'s via the following lemma.  It is important to point out that $z=0$ and $z=\infty$ correspond to the two branches at $x=\infty$.  The expansion in $z$ is around $z=0$ while the expansion in $x$ is around the other branch $z=\infty$.  This is essentially due to the need for both expansions to have positive coefficients.
\begin{lemma}\label{ntom} For the curve $x=z+1/z$, $y=y(z)$
\begin{equation}\label{ntomformula}
M^g_n(b_1,\dots,b_n)=\sum_{l_i> \frac{b_i}{2}}^{b_i}N^g_n(2l_1-b_1,...,2l_n-b_n)\prod_{i=1}^n (2l_i-b_i)\binom{b_i}{l_i}
%&=(-1)^n\prod_{i=1}^n D_iB_{\infty,b_i}(D_i)N^g_n(0,\dots,0)
\end{equation}
\end{lemma}
\begin{proof} Extract the coefficients of a local expansion of $\omega^g_n$ in $x_i^{-1}$ by taking residues.
\begin{align*} 
M^g_n(b_1,...,b_n)&:=(-1)^n\res{x_1=\infty}...\res{x_n=\infty} x_1^{b_1}... x_n^{b_n}\cdot\omega^g_n(z_1,...,z_n)\\
&=(-1)^n\res{z_1=\infty}... \res{z_n=\infty} x_1^{b_1}... x_n^{b_n}\cdot\omega^g_n(z_1,...,z_n)\\
&=\res{z_1=0}... \res{z_n=0} x_1^{b_1}... x_n^{b_n}\cdot\omega^g_n(z_1,...,z_n)\quad(\Leftarrow\omega^g_n(1/z_1,...,1/z_n)=(-1)^n\omega^g_n(z_1,...,z_n)\ )\\
&=\prod_{i=1}^n \res{z_i=0} \left(\frac{1}{z_i}+z_i\right)^{b_i}\sum_{k_1,...,k_n=1}^\infty N^g_n(k_1,...,k_n)\prod_{i=1}^nk_iz_i^{k_i-1}dz_i\\
&=\prod_{i=1}^n \res{z_i=0}\sum_{l_1,...,l_n=0}^{b_i}\sum_{k_1,...,k_n=1}^\infty N^g_n(k_1,...,k_n)\prod_{i=1}^nk_i\binom{b_i}{l_i}z_i^{b_i-2l_i+k_i-1}dz_i\\
&= \sum_{l_i> \frac{b_i}{2}}^{b_i}N^g_n(2l_1-b_1,...,2l_n-b_n)\prod_{i=1}^n (2l_i-b_i)\binom{b_i}{l_i}.
%&=(-1)^n\prod_{i=1}^n D_iB_{\infty,b_i}(D_i)N^g_n(0,\dots,0)\\
%&=\prod_{i=1}^n\frac{\partial}{\partial t_{\infty,b_i}} F^g
%&=\sum_{m_1,...,m_n=1}^{l_i} N^g_n(2m_1,...,2m_p,2m_{p+1}+1,...,2m_n+1)\prod_{i=1}^p 2m_i \binom{2l_i}{m_i+l_i}\prod_{i=p+1}^n (2m_i+1) \binom{2l_i+1}{m_i+l_i+1}\\
\end{align*}
\end{proof}
Analogous to the notation $N^g_{n,k}(b_1,\dots,b_n)$ which is the polynomial expression for $N^g_n$ corresponding to the first $k$ variables being odd, since the sum (\ref{ntomformula}) respects parity, we define $M^g_{n,k}(b_1,\dots,b_n)$ to be the expression for $M^g_n$ with $k$ odd variables, obtained by summing $N^g_{n,k}$ terms.
\begin{lemma}\label{cortransform} $M^g_{n,k}(b_1,\dots,b_n)$ can be obtained from $N^g_{n,k}(b_1,\dots,b_n)$ via the term-by-term transform on monomials 
\begin{equation}\label{transform}
b_1^{2\alpha_1}\cdots b_n^{2\alpha_n}\mapsto \prod_{i=1}^kb_i\binom{b_i-1}{\frac{b_i-1}{2}}q_{\alpha_i}\left(\frac{b_i-1}{2}\right)\prod_{i=k+1}^n \frac{b_i}{2}\binom{b_i}{\frac{b_i}{2}}p_{\alpha_i}\left(\frac{b_i}{2}\right)
\end{equation}
where $q_\alpha(n)$ and $p_\alpha(n)$ are polynomials of degree $\alpha$ satisfying the recurrences 
\begin{align}
p_{\alpha+1}(n)&=4n^2(p_\alpha(n)-p_\alpha(n-1))+4np_\alpha(n-1)\label{prec},& p_0(n)&=1\\
 q_{\alpha+1}(n)&=4n^2(q_\alpha(n)-q_\alpha(n-1))+(4n+1)q_\alpha(n-1)\label{qrec},& q_0(n)&=1
\end{align}
\end{lemma}
\begin{proof} As the sum (\ref{ntomformula}) is over all combinations of $l_i$ for each $i$, for monomial terms of several variables we can factorise
\begin{equation}
\sum_{l_i> \frac{b_i}{2}}^{b_i}\prod_{i=1}^n (2l_i-b_i)^{2\alpha_i+1}\binom{b_i}{l_i}=\prod_{i=1}^n \sum_{l_i> \frac{b_i}{2}}^{l_i}(2l_i-b_i)^{2\alpha_i+1}\binom{b_i}{l_i}
\end{equation}
to reduce the problem to the one variable case. For different parities $b=2n$ and $b=2n+1$, the sums become
\begin{equation}
\sum_{l> \frac{b}{2}}^{b}(2l-b)^{2\alpha+1}\binom{b}{l}= \begin{cases} \sum_{l=0}^n\binom{2n}{n-l}(2l)^{2\alpha+1}, &b=2n\\
\sum_{l=0}^n \binom{2n+1}{n-l}(2l+1)^{2\alpha+1},  & b=2n+1
\end{cases}
\end{equation}
after exchanging $l\mapsto n-l$.  From \cite{Tue}, the sum $$ \tilde{p}_\alpha(n):=\sum_{l=0}^n\binom{2n}{n-l}(2l)^{2\alpha+1} $$ satisfies the three term recurrence $$\tilde{p}_{\alpha+1}(n)=4n^2\tilde{p}_\alpha(n)-8n(2n-1)\tilde{p}_\alpha(n-1), \indent \tilde{p}_0(n)\binom{2n}{n}.$$  Letting $\tilde{p}_\alpha(n)=p_\alpha(n)n\binom{2n}{n}$ gives the required recursion (\ref{prec}) for $p_\alpha$.  The proof for the odd case proceeds in the same manner, this time starting from the three term recursion $$\tilde{q}_{\alpha+1}(n)=(2n+1)^2\tilde{q}_\alpha(n)-8n(2n+1)\tilde{q}_\alpha(n-1), \indent \tilde{q}_0(n)=(2n+1)\binom{2n}{n}.$$
\end{proof}
The first few transformation polynomials (in the form useful for (\ref{transform})) are 
\begin{align*}
p_0(\frac{b}{2})&=1& q_0(\frac{b-1}{2})&=1\\
p_1(\frac{b}{2})&=2b & q_1(\frac{b-1}{2})&=2b-1\\
p_2(\frac{b}{2})&=8b(b-1)& q_2(\frac{b-1}{2})&=8b^2-12b+5\\
p_3(\frac{b}{2})&=16b(3b^2-8b+6)& q_3(\frac{b-1}{2})&=48b^3-152b^2+166b-61.\\
\end{align*}
The $p_\alpha$ and $q_\alpha$ are generalisations of the Gandhi polynomials, related to the Dumont-Foata polynomials.  See \cite{Tue} and the references therein for a survey and properties of these topics.

\begin{proposition}\label{Mpoly}For $2g-2+n>0$, the coefficients $M^g_{n,k}$ in the expansion of the Eynard-Orantin invariants of (\ref{generalC}) about $x=\infty$ can be expressed as
\begin{equation}  \label{Mprod}
M^g_{n,k}=\prod_{i=1}^k b_i\binom{b_i-1}{\frac{b_i-1}{2}}\prod_{i={k+1}}^n\frac{b_i}{2}\binom{b_i}{\frac{b_i}{2}}m^g_{n,k}(b_1,\dots,b_n),
\end{equation} where $m^g_{n,k}(b_1,\dots,b_n)$ is a polynomial of degree $3g-3+n$, symmetric in variables of the same parity, with coefficient $v_{\beta}$ of $b_1^{\beta_1}\cdots b_n^{\beta_n}$ given by 
\begin{equation}  \label{Mcoeff}
v_{\beta}=\frac{y'(1)^{2-2g-n}+(-1)^{k}y'(-1)^{2-2g-n}}{2^{2g-2+n}}\int_{\overline{\modm}_{g,n}}\psi_1^{\beta_1}...\psi_n^{\beta_n}
\end{equation}
for $|\beta|=3g-3+n$.
\end{proposition}
\begin{proof}
Expand the Eynard-Orantin invariants about $z=0$, and apply Lemmas \ref{zexpand}, \ref{Nintersection} and \ref{cortransform} to get expressions for $M^g_n$.  To prove the proposition we need the polynomials $p_\alpha(b/2)$ and $q_\alpha\left((b-1)/2\right)$ used in the transformation (\ref{transform}) to have leading order coefficients $\alpha!2^{\alpha}$.  

By induction, suppose $p_\alpha(n)$ has leading coefficient $\alpha!2^{2\alpha}$.  Using the recursion for $p_{\alpha+1}(n)$, the leading part of $p_{\alpha+1}(n)$ is:
\begin{align*}
\alpha!2^{2\alpha}\left(4n^2(n^{\alpha}-(n-1)^\alpha)\right)&+4\alpha!2^{2\alpha}n^{\alpha+1}+O(n^{\alpha})\\
&=\alpha!2^{2\alpha}\left(4n^2(\alpha n^{\alpha-1})+4n(n^\alpha)\right)+O(n^{\alpha})\\
&=(\alpha+1)!2^{2\alpha+2}n^{\alpha+1}+O(n^{\alpha})
\end{align*}
Similarly, the recursion for $q_{\alpha+1}(n)$ shows that the leading part of $q_{\alpha+1}(n)$ is:
$$\alpha!2^{2\alpha}\left(4n^2(\alpha n^{\alpha-1})+(4n+1)n^\alpha\right) +O(n^{\alpha-1})=(\alpha+1)!2^{2\alpha+2}n^{\alpha+1} +O(n^{\alpha})$$ so that all transformation polynomials have the required leading order coefficients.
\end{proof}

\subsection{Divisor and string equations.}
For the remainder of the paper we specialise to the curve (\ref{eq:lnz}):
\begin{equation*} 
C=\begin{cases}x=z+1/z\\ 
y=\ln{z}.
\end{cases}
\end{equation*}
The recursions (\ref{Mdivisor}) and (\ref{Mstring}) below use the terms {\em divisor} and {\em string} equations which anticipate the corresponding recursions (\ref{GWdivisor}) and (\ref{GWstring}) satisfied by Gromov-Witten invariants.
\begin{theorem}\label{th:Mgenus0}The coefficients $M^g_n$ in the expansion of the Eynard-Orantin invariants of (\ref{eq:lnz}) about $x=\infty$ satisfy the divisor and string equations.  For $2d=2-2g-n+\sum_{i=1}^nb_i$,
\begin{align}
M^g_{n+1}(b_1,\dots,b_n,1)&=dM^g_n(b_1,\dots,b_n)\label{Mdivisor}\\
M^g_{n+1}(b_1,\dots,b_n,0)&=\sum_{i=1}^nb_iM^g_{n}(b_1,\dots,b_i-1,\dots,b_n)\label{Mstring}
\end{align}
Where 
\begin{equation*}
M^g_{n+1,k}(b_1,\dots,b_n,0):=\prod_{i=1}^k b_i\binom{b_i-1}{\frac{b_i-1}{2}}\prod_{i={k+1}}^n\frac{b_i}{2}\binom{b_i}{\frac{b_i}{2}}m^g_{n+1,k}(b_1,\dots,b_n,0).
\end{equation*}
These uniquely determine all genus zero terms and, together with the top degree terms known from Proposition~\ref{Mpoly}, determine all genus one terms.
\end{theorem}
\begin{proof}  In the following we use $ \int_0^z\omega^g_{n+1}(z_S,z')$ which is well-defined (independently of the choice of path) since the residues of $\omega^g_{n+1}$ are zero.  The calculations below will not be sensitive to the constant term arising from the choice of intial point 0 in the integral.  To prove equation (\ref{Mdivisor}), we use the string and dilaton equations (\ref{EOstring})-(\ref{EOdilaton}).
\begin{align*}
\sum_{\alpha=\pm1}\res{z=\alpha}\left( yx-\int ydx\right)\omega^g_{n+1}(z_S,z)&=\sum_{\alpha=\pm1}\res{z=\alpha}\left( z\text{ln}(z)+\frac{\text{ln}(z)}{z}-\int^z_{z_0} \text{ln}(t)(1-\frac{1}{t^2})dt\right)\omega^g_{n+1}(z_S,z)\\
&=\sum_{\alpha=\pm1}\res{z=\alpha}\left( z-\frac{1}{z}+c\right)\omega^g_{n+1}(z_S,z)\\
&=-\sum_{\alpha=0,\infty}\res{z=\alpha}\left(z-\frac{1}{z}\right)\omega^g_{n+1}(z_S,z)\\
&=-2\res{z=\infty}z\omega^g_{n+1}(z_S,z)\indent\indent \text{\ [Since $\omega^g_{n+1}(z_S,1/z)=-\omega^g_{n+1}(z_S,z)$]}\\
&=-2\res{z=0}x(z)\omega^g_{n+1}(z_S,z)\indent\indent \text{\ [Add residue free term]}\\
&=2\sum_{b_1,\dots,b_n=0}^\infty \frac{M^g_{n+1}(b_1,\dots,b_n,1)}{x_1^{b_1+1}\cdots x_n^{b_n+1}}dx_1\cdots dx_n.
\end{align*}
While the right hand side of (\ref{EOstring})-(\ref{EOdilaton}) gives 
\begin{align*}
-\sum_{i=1}^n\partial_{z_i}\left(\frac{x(z_i)\omega^g_n(z_S)}{dx(z_i)}\right)&-(2g-2+n)\omega^g_n(z_S)
\\ &=-\sum_{i=1}^n\partial_{x_i}\left(x_i\sum_{b_1,\dots,b_n=0}^\infty \frac{M^g_{n}(b_1,\dots,b_n)}{x_1^{b_1+1}\cdots x_n^{b_n+1}}dx_1\cdots d\hat{x}_i\cdots dx_n\right)\\&\quad-(2g-2+n)\sum_{b_1,\dots,b_n=0}^\infty \frac{M^g_{n}(b_1,\dots,b_n)}{x_1^{b_1+1}\cdots x_n^{b_n+1}}dx_1\cdots dx_n\\
&=(\sum_{i=1}^n b_i+2-2g-n)\sum_{b_1,\dots,b_n=0}^\infty \frac{M^g_{n}(b_1,\dots,b_n)}{x_1^{b_1+1}\cdots x_n^{b_n+1}}dx_1\cdots dx_n.
\end{align*}
Equating coefficients and using $2d=2-2g-n+\sum_{i=1}^nb_i$ gives (\ref{Mdivisor}) as required.

To prove (\ref{Mstring}) take $m=0$ in (\ref{EOstring}). When expanded around $x_i=\infty$ the RHS gives
\begin{align*}
-\sum_{i=1}^n\partial_{z_i}\left(\frac{\omega^g_n(z_1,\dots,z_n)}{dx(z_i)}\right)&=-\sum_{i=1}^n\partial_{x_i}\left(\sum_{b_1,\dots,b_n=0}^\infty \frac{M^g_{n}(b_1,\dots,b_n)}{x_1^{b_1+1}\cdots x_n^{b_n+1}}dx_1\cdots d\hat{x}_i\cdots dx_n\right)\\
&=\sum_{b_1,\dots,b_n=0}^\infty \frac{\sum_{i=1}^n (b_i+1)M^g_{n}(b_1,\dots,b_n)}{x_1^{b_1+1}\cdots x_i^{b_i+2}\cdots x_n^{b_n+1}}dx_1\cdots dx_n\\
&=\sum_{b_1,\dots,b_n=0}^\infty \frac{\sum_{i=1}^n b_iM^g_{n}(b_1,\dots,b_i-1,\dots,b_n)}{x_1^{b_1+1}\cdots x_n^{b_n+1}}dx_1\cdots dx_n.
\end{align*}
Where $d\hat{x}_i$ denotes $dx_i$ missing from the first term.  For the LHS, we need the following lemma.
\begin{lemma}\label{intseries} Let $F(z)=\sum_{n=1}^\infty p(n)z^n$ for a quasi-polynomial $p(n)$.  Then $F(z)$ is a meromorphic function on $\bp^1$, analytic at 0 and $\infty$, satisfying $F(\infty)-F(0)=-p(0)$.
\end{lemma}
\begin{proof}  Recall that $p(n)$ is quasi-polynomial in $n$ if it is polynomial on each coset of a sublattice $m\bz\subset\bz$, i.e. it is represented by $m$ polynomials $p_a(n)$, $a=1,...,m$, for $n\equiv a(m)$, and $p(0):=p_m(0)$.

Decompose $F(z)$ into
\[ F(z)=\sum_{n=1}^\infty p(n)z^n=\sum_{a=1}^m\ \sum_{0<n\equiv a(m)}p_a(n)z^n\]
and further decompose $p_a(n)$ into linear combinations of monomials $n^k$.  Then
\[\sum_{0<n\equiv a(m)}n^kz^n=\left( z\frac{d}{dz}\right)^k\sum_{0<n\equiv a(m)}z^n=\left( z\frac{d}{dz}\right)^k\frac{z^a}{1-z^m}\]
which vanishes at $z=\infty$, since the denominator has greater degree than than the numerator, except when $k=0$ and $a=m$, where at $z=\infty$ it evaluates to $-1$.
\end{proof}
The LHS of (\ref{EOstring}) now becomes
\begin{align*}
\sum_{\alpha=\pm1}\res{z=\alpha}\text{ln}(z)\omega^g_{n+1}(z_S,z)&=-\sum_{\alpha=\pm1}\res{z=\alpha}\frac{dz}{z}\int_0^z\omega^g_{n+1}(z_S,z')~~~~~\indent\indent \text{\ [Integrating by parts]}\\
&=\sum_{\alpha=0,\infty}\res{z=\alpha}\frac{dz}{z}\int_0^z\omega^g_{n+1}(z_S,z')\\
&=\res{z=\infty}\frac{dz}{z}\int_0^z\omega^g_{n+1}(z_S,z')\indent\indent \text{\ [Analytic at $z=0$]}\\
&=-\int_0^\infty \omega^g_{n+1}(z_S,z)\\
&=\sum_{k_1,\dots,k_n=1}^\infty k_1\cdots k_n N^g_{n+1}(k_S,0)z_S^{k_S-1}dz_S \indent\indent \text{\ [Lemmas \ref{zexpand}, \ref{intseries}]}\\
&=\sum_{b_1,\dots,b_n=0}^\infty \frac{\prod_{i=1}^k b_i\binom{b_i-1}{\frac{b_i-1}{2}}\prod_{i={k+1}}^n\frac{b_i}{2}\binom{b_i}{\frac{b_i}{2}}m^g_{n+1,k}(b_S,0)}{x_1^{b_1+1}\cdots x_n^{b_n+1}}dx_1\cdots dx_n
\end{align*}
where in the final step we have changed the first $n$ variables from expansions in $z$ to $x$ using the transform from Lemma~\ref{cortransform}.  The last variable contains only a constant term which remains unchanged under this transform.

\begin{lemma}  \label{th:sympol}
Let $f_k(t_1,...,t_n)$ be a polynomial symmetric in the variables $t_1,...,t_k$ and also symmetric in the variables $t_{k+1},...,t_n$.  Evaluation at the two variables $f_k(a,t_2,...,t_n)$ and $f_k(t_1,...,t_{n-1},b)$ for any $a$, $b$ determines any such $f_k$ of degree less than $n$ and if the degree of $f_k$ equals $n$ it determines $f_k$ up to a constant. 
\end{lemma}
\begin{proof}
Suppose $g_k(t_1,...,t_n)$ were another polynomial of the same degree as $f_k$, symmetric in variables of the same parity satisfying $g_k(a,t_2,...,t_n)=f_k(a,t_2,...,t_n)$ and $g_k(t_1,...,t_{n-1},b)=f_k(t_1,...,t_{n-1},b)$.  Define $h_k(t_1,...,t_n)=g_k(t_1,...,t_n)-f_k(t_1,...,t_n)$.  Then $h_k(a,t_2,...,t_n)=0=h_k(t_1,...,t_{n-1},b)$.  By symmetry 
$$h_k(t_1,...,t_n)=\prod_{i=1}^k(t_i-a)\prod_{i=k+1}^n(t_i-b)\tilde{h}_k(t_1,...,t_n)$$ for some other polynomial $\tilde{h}_k$. 
If $\deg f_k=\deg h_k<n$ then $\tilde{h}_k\equiv 0$ and $g_k(t_1,...,t_n)=f_k(t_1,...,t_n)$.  If $\deg f_k=\deg h_k$ then $\tilde{h}_k\equiv \lambda$ is constant and $g_k(t_1,...,t_n)=f_k(t_1,...,t_n)+\lambda\prod_{i=1}^k(t_i-a)\prod_{i=k+1}^n(t_i-b)$.
\end{proof}
Note that the lemma makes sense and remains true if $k=0$ or $n$.

To complete the proof of Theorem~\ref{th:Mgenus0} we need to show that the divisor and string equations determine the genus zero and genus one Eynard-Orantin invariants.

For any $k=1,...,n$, the divisor equation (\ref{Mdivisor}) and string equation (\ref{Mstring}) allow us to compute $m^g_{n+1,k}(b_1,\dots,b_n,0)$ and $m^g_{n+1,k}(1,b_2,\dots,b_{n+1})$ from smaller $m^g_{n,k}$.  (We have assumed that $b_1$ is odd and $b_{n+1}$ is even.  If $k=0$ or $n+1$, the string and divisor equation respectively are alone sufficient to determine $m^0_{n+1,k}$ using precisely the same argument.)   For $g=0$ and each $k$, $m^0_{n+1,k}(b_1,\dots,b_{n+1})$ is a polynomial of degree $n-2$, symmetric in variables of the same parity.  Hence Lemma~\ref{th:sympol} shows that $m^0_{n+1,k}(b_1,\dots,b_{n+1})$ is uniquely determined from smaller $m^0_{n,k}$.

For $g=1$ and each $k$, $m^1_{n+1,k}(b_1,\dots,b_{n+1})$ is a polynomial of degree $n+1$, symmetric in variables of the same parity.  Hence Lemma~\ref{th:sympol} shows that the string and dilaton equations determine $m^1_{n+1,k}(b_1,\dots,b_{n+1})$ from smaller $m^1_{n,k}$ up to $\lambda\cdot\prod_{i=1}^k(b_i-1)\prod_{i=k+1}^{n+1}b_i$.  The constant $\lambda$ can be determined from Proposition~\ref{Mpoly} which gives the coefficients of all top degree terms in terms of intersection numbers on $\overline{\modm}_{g,n}$.  In particular, the coefficient of $b_1...b_{n+1}$ in $m^1_{n+1,k}(b_1,\dots,b_{n+1})$ is $2^{1-n}\langle\tau_1^n\rangle=\displaystyle\frac{2^{1-n}(n-1)!}{24}$.
\end{proof}

\section{Gromov-Witten invariants} \label{sec:GW}

\subsection{The moduli space of stable maps} 
Let $X$ be a projective algebraic variety and consider $(C,x_1,\dots,x_n)$ a connected smooth curve of genus $g$ with $n$ distinct marked points.  For $\beta \in H_2(X,\bz)$ the moduli space of maps $\modm^g_n(X,\beta)$ consists of morphisms 
$$\pi: (C,x_1,\dots,x_n)\rightarrow X$$
satisfying $\pi_\ast [C]=\beta$ quotiented by isomorphisms of the domain $C$ that fix each $x_i$.  The moduli space has a compactification $\overline{\modm}^g_n(X,\beta)$ given by the moduli space of stable maps:  the domain $C$ is a connected nodal curve; the distinct points $\{x_1,\dots,x_n\}$ avoid the nodes; any genus zero irreducible component of $C$ with fewer than three distinguished points (nodal or marked) must not be collapsed to a point; any genus one irreducible component of $C$ with no marked point must not be collapsed to a point.  The moduli space of stable maps has irreducible components of different dimensions but its expected or virtual dimension is 
\[\dim\overline{\modm}^g_n(X,\beta)=\langle c_1(X),\beta\rangle +(\dim X-3)(1-g)+n.\]

\subsubsection{Cohomology on $\overline{\modm}^g_n(X,\beta)$}
Let $\mathcal{L}_i$ be the cotangent bundle over the $i$th marked point and $\psi_i\in H^2(\overline{\modm}^g_n(X,\beta),\bq)$ be the first chern class of $\mathcal{L}_i$.

For $i=1,\dots,n$ there exist evaluation maps:
\begin{equation}
ev_i:\overline{\modm}^g_n(X,\beta)\longrightarrow X, ~ ev_i(\pi)=\pi(x_i)
\end{equation}
and classes $\gamma\in H^*(X,\bz)$ pull back to classes in $H^*(\overline{\modm}^g_n(X,\beta),\bq)$
\begin{equation}
ev_i^\ast:H^*(X,\bz)\longrightarrow H^*(\overline{\modm}^g_n(X,\beta),\bq)
\end{equation}

Gromov-Witten theory involves integrating cohomology classes, often called descendent classes, of the form
$$\tau_{b_i}(\gamma)=\psi_i^{b_i}ev^\ast_i(\gamma).$$

These are integrated against the \textit{virtual fundamental class}, $[\overline{\modm}^g_n(X,\beta)]^{vir}$, the existence and construction of which is highly nontrivial.  

Gromov-Witten invariants quite generally satisfy divisor, string and dilaton equations \cite{WitTwo} and topological recursion relations arising from relations on the moduli space of curves $\overline{\modm}_{g,n}$, \cite{GetTop}.  We will write these relations only in the special case when the target is $\bp^1$.

\subsection{Specialising to $\bp^1$}
We now only consider the specific case of Gromov-Witten invariants of $\bp^1$.   Let $\omega\in H^2(\bp^1,\bq)$ be the Poincare dual class of a point and $1\in H^0(\bp^1,\bq)$ Poincare dual to the fundamental class.  We consider the invariants  
\begin{equation}\label{gropoint1}
\left\langle \prod_{i=1}^l\tau_{b_i}(1)  \prod_{i=l+1}^n\tau_{b_i}(\omega)\right\rangle ^{g}_{d}=\int_{[\overline{\modm}^g_n(\bp^1,d)]^{vir}} \prod_{i=1}^l\psi_i^{b_i} \prod_{i=l+1}^n\psi_i^{b_i}ev_i^\ast(\omega)
\end{equation}
where we consider only connected invariants and (\ref{gropoint1}) is defined to be zero unless $\sum_{i=1}^n b_i=2g-2+2d+l$.  In our notation, often either $g$ or $d$ will be missing when clear, since the dimension restraints define one from the other.  Our main interest is the case $l=0$, known as the (connected) stationary Gromov-Witten theory of $\bp^1$ since the images of the marked points are fixed.

We collect here a few properties of Gromov-Witten invariants of $\bp^1$ needed here.  We recommend reading \cite{OPaGro}, \cite{OPaGro2} and \cite{OPaVir}, for a thorough treatment of this case.  We use the following divisor, string and dilaton equations \cite{WitTwo} principally for stationary Gromov-Witten invariants.  For $2d=2-2g+\sum_{i=1}^n b_i$, and $\alpha_i\in \{1,\omega\}$
\begin{align}
{\rm\bf divisor\ equation}\quad\left\langle \tau_0(\omega)\tau_{b_1}(\alpha_1)\cdots \tau_{b_n}(\alpha_n)\right\rangle_d&=d\left\langle \tau_{b_1}(\alpha_1)\cdots \tau_{b_n}(\alpha_n)\right\rangle_d\label{GWdivisor}\\
&\quad+\sum_{i=1}^n\left\langle \tau_{b_1}(\alpha_1)\cdots\tau_{b_i-1}(\alpha_i\cup\omega)\cdots \tau_{b_n}(\alpha_n)\right\rangle_d\nonumber\\
{\rm\bf string\ equation}\quad\ \ \left\langle \tau_0(1)\tau_{b_1}(\alpha_1)\cdots \tau_{b_n}(\alpha_n)\right\rangle_d&=\sum_{i=1}^n\left\langle \tau_{b_1}(\alpha_1)\cdots\tau_{b_i-1}(\alpha_1)\cdots \tau_{b_n}(\alpha_n)\right\rangle_d\label{GWstring}\\
{\rm\bf dilaton\ equation}\quad\left\langle \tau_1(1)\tau_{b_1}(\alpha_1)\cdots \tau_{b_n}(\alpha_n)\right\rangle^g&=(2g-2+n)\left\langle \tau_{b_1}(\alpha_1)\cdots \tau_{b_n}(\alpha_n)\right\rangle^g\label{GWdilaton}
\end{align}
where we define $\tau_{b}(0)=0$.  Consider the generating function for descendent classes
\begin{equation*}
F=\exp \sum_{b=0}^\infty\left(t_b \tau_b(\omega)+s_b \tau_b(1)\right).
\end{equation*}
For $\alpha_i\in \{1,\omega\}$ the genus zero topological recursion \cite{WitTwo} is 
\begin{align}\label{fullzerotoprec}
\left\langle \tau_{b_1}(\alpha_1) \tau_{b_2}(\alpha_2) \tau_{b_3}(\alpha_3) F\right\rangle^0&=\left\langle \tau_0(1)\tau_{b_1-1}(\alpha_1)F\right\rangle^0 \left\langle \tau_0(\omega)\tau_{b_2}(\alpha_2)\tau_{b_3}(\alpha_3)F\right\rangle^0 \\
\nonumber &+\left\langle \tau_0(\omega)\tau_{b_1-1}(\alpha_1)F\right\rangle^0 \left\langle \tau_0(1)\tau_{b_2}(\alpha_2)\tau_{b_3}(\alpha_3)F\right\rangle^{0}
\end{align}
and the genus one topological recursion is
\begin{align}\label{fullonetoprec}
\left\langle \tau_{b_1}(\alpha_1)F\right\rangle^1=\left\langle \tau_0(1)\tau_{b_1-1}(\alpha_1)F\right\rangle^0\left\langle \tau_0(\omega)F\right\rangle^1 &+\left\langle \tau_0(\omega)\tau_{b_1-1}(\alpha_1)F\right\rangle^0\left\langle \tau_0(1)F\right\rangle^1\\ \nonumber &+\frac{1}{12}\left\langle\tau_0(1)\tau_0(\omega) \tau_{b_1-1}(\alpha_1)F\right\rangle^0.
\end{align}

In \cite{OPaGro}, Okounkov and Pandharipande show that for Gromov-Witten invariants that allow disconnected domains (denoted by the superscript $^{\bullet}$) the following relation holds:
\begin{equation}  \label{eq:gwplanch}
\big\langle \prod_{i=1}^n \tau_{b_i}(\omega)\big\rangle_d^{\bullet } = \sum_{|\lambda |=d} \Big(\frac{\dim\lambda}{d!}\Big)^2\prod_{i=1}^n \frac{\textbf{p}_{b_i+1}(\lambda)}{(b_i+1)!}
\end{equation}
where the sum is over all partitions of $d$ and for a partition $\lambda$,  $\textbf{p}_{k}(\lambda)$ is the shifted symmetric power sum defined by
$$\textbf{p} _{k}(\lambda)=\sum_{i=1}^\infty \big[(\lambda_i-i+\frac{1}{2})^{k}-(-i+\frac{1}{2})^k\big]+(1-2^{-k})\zeta(-k).$$

\section{Proof of Theorem \ref{th:main}}\label{sec:proofofmain}
The strategy of the proof of Theorem \ref{th:main} will be to use recursions to uniquely determine both the Eynard-Orantin invariants and the Gromov-Witten invariants of $\bp^1$ and compare.  The obvious candidates for the genus 0 and 1 Eynard-Orantin invariant are the divisor and string equations, (\ref{Mdivisor}) and (\ref{Mstring}).  The genus 0 and 1 Gromov-Witten invariants of $\bp^1$ are determined by the topological recursion relations (\ref{fullzerotoprec}) and (\ref{fullonetoprec}).  However, the two sets of recursion relations are not compatible, so we first produce new recursion relations for the stationary Gromov-Witten invariants of $\bp^1$, given in Section~\ref{sec:statgw}, which are interesting in their own right, and serve our purposes here. 

\subsection{Polynomial behaviour of Gromov-Witten invariants}  \label{sec:gwpol}
 We begin by proving the following weaker version of Theorem~\ref{th:GWquasi}.
\begin{proposition} \label{GWquasi}For $g=0$ and 1, the stationary Gromov-Witten invariants are of the form
\begin{equation}\label{eq:quasi1}
\left\langle \prod_{i=1}^k\tau_{2u_i}(\omega)\prod_{i=k+1}^n\tau_{2u_i-1}(\omega)\right\rangle^g=\frac{u_{k+1}\cdots u_n}{\prod_{i=1}^n u_i!^2}p^g_{n,k}(u_1,\dots,u_n)
\end{equation}
where $p^g_{n,k}(u_1,\dots,u_n)$ is a polynomial of degree $3g-3+n$ in the $u_i$'s, symmetric in the first $k$ and the last $n-k$ variables.
\end{proposition}
\begin{proof}
We prove this by induction using the topological recursion relations for genus zero and genus one Gromov-Witten invariants.  
%\begin{align}\label{onetoprec}
%\left\langle \tau_{b_1}(\omega) F\right\rangle^{g=1}&=\left\langle \tau_0(1)\tau_{b_1-1}(\omega) F\right\rangle^{g=0}\left\langle \tau_{0}(\omega) F\right\rangle^{g=1}\\\nonumber &+\left\langle \tau_0(\omega)\tau_{b_1-1}(\omega) F\right\rangle^{g=0}\left\langle \tau_{0}(1) F\right\rangle^{g=1}
%+\frac{1}{12}\left\langle \tau_0(1)\tau_0(\omega)\tau_{b_1-1}(\omega) F\right\rangle^{g=0}.
%\end{align}

{\bf Genus zero case.}
\subsubsection{Initial cases} The recursion (\ref{fullzerotoprec}) can be used along with the string and divisor equations to explicitly find expressions for genus zero $1,2$ and 3-point invariants.  The one point invariants \cite{OPaGro}
$$\left\langle \tau_{2u}(\omega)\right\rangle^{0} =\frac{1}{(u+1)!^2}=\frac{1}{u!^2}\frac{1}{(u+1)^2}.$$
Let $\alpha_i=\omega$, $F=1$, $b_1=2u_1$, $b_2=2u_2$ and $b_3=0$.  Repeated application of the string and divisor equations gives
\begin{align*}\left\langle \tau_{2u_1}(\omega)\tau_{2u_2}(\omega)\tau_{0}(\omega)\right\rangle^{0} &=(u_2+1)(u_2+1)\left\langle \tau_{2u_1-2}(\omega)\right\rangle^{0} \left\langle \tau_{2u_2}(\omega)\right\rangle^{0}+0 \\
&=\frac{(u_2+1)^2}{u_1!^2(u_2+1)!^2}=\frac{1}{u_1!^2u_2!^2}\\
\Rightarrow\quad\quad\quad \indent \left\langle \tau_{2u_1}(\omega)\tau_{2u_2}(\omega)\right\rangle^{0}&=\frac{1}{u_1!^2u_2!^2}\frac{1}{u_1+u_2+1}.
\end{align*}
Similarly, $\alpha_i=\omega$, $F=1$, $b_1=2u_1-1$, $b_2=2u_2-1$ and $b_3=0$ gives
$$\indent \left\langle \tau_{2u_1-1}(\omega)\tau_{2u_2-1}(\omega)\right\rangle^{0}=\frac{u_1u_2}{u_1!^2u_2!^2}\frac{1}{u_1+u_2}.$$
As mentioned in the introduction, the one-point and two-point functions still satisfy Proposition~(\ref{GWquasi}) if we interpret degree -2 and -1 polynomials to mean the reciprocal of degree 2 and degree 1  polynomials.

We can now use (\ref{fullzerotoprec}), the string and divisor equations to compute the initial step of the induction - the three point invariants:
\begin{align}\label{03GW}
\langle \tau_{2u_1}(\omega)\tau_{2u_2}(\omega)\tau_{2u_3}(\omega) \rangle^{0} &= \frac{1}{u_1!^2u_2!^2u_3!^2}\\
\nonumber \langle \tau_{2u_1}(\omega)\tau_{2u_2-1}(\omega)\tau_{2u_3-1}(\omega) \rangle^{0}&= \frac{u_2u_3}{u_1!^2u_2!^2u_3!^2}
\end{align}

Before we apply the inductive step, we need the following lemma.
\begin{lemma}Proposition~\ref{GWquasi} can be extended to include $\tau_0(1)$ terms. 
\end{lemma}
\begin{proof}This uses the string equation (\ref{GWstring}).  Suppose Proposition~\ref{GWquasi} holds for the right hand side of the string equation (\ref{GWstring}).  Then we must check that the left hand side is the required degree polynomial.  Let $K=\{1,\dots,k\}$ and $J=\{{k+1},\dots n\}$. The equation can be written
%\begin{align*}
%\left\langle \tau_0(1)\prod_{i=1}^k\tau_{2b_i-1}(\omega)\cdots \prod_{i=k+1}^n\tau_{2b_i}(\omega)\right\rangle_d&=\sum_{i=1}^k \frac{b_1 \cdots b_k}{b_1!^2\cdots (b_i-1)!^2\cdots b_n!^2}p^g_{n,k-1}(b_1,\dots,b_k,b_i-1,\dots b_n)\\
%&+\sum_{i=k+1}^n \frac{b_1\cdots b_n b_i}{b_1!^2\cdots b_n!^2}p^g_{n,k+1}(b_i,b_1,\dots,b_k,\dots,\hat{b}_{i},\dots,b_n)\\
%&=\frac{b_1\cdots b_k}{\prod_{i=1}^n b_i!^2}\Big(\sum_{i=1}^k b_ip^g_{n,k-1}(b_1,\dots, \hat{b}_i \dots,b_k,b_i,\dots,b_n)\\
%&+\sum_{i=k+1}^n b_ip^g_{n,k+1}(b_i,b_1,\dots,b_k,\dots,\hat{b}_{i},\dots,b_n)\Big)
%\end{align*}
\begin{align*}
\left\langle \tau_0(1)\prod_{i=1}^k\tau_{2u_i}(\omega) \prod_{i=k+1}^n\tau_{2u_i-1}(\omega)\right\rangle^g&=\sum_{i=1}^k \frac{u_iu_{k+1} ...u_n}{u_1!^2... u_n!^2}p^g_{n,k-1}(u_{K\setminus i},u_i,u_J)\\
&\quad+\sum_{i=k+1}^n \frac{u_{k+1}...\hat{u}_i ...u_n}{u_1!^2...(u_i-1)!^2...u_n!^2}p^g_{n,k+1}(u_K,u_i-1,u_{J\setminus i})\\
&=\frac{u_{k+1}\cdots u_n}{\prod_{i=1}^n u_i!^2}\Big(\sum_{i=1}^k u_ip^g_{n,k-1}(u_{K\setminus i},u_i,u_J)
+\hspace{-2mm}\sum_{i=k+1}^n\hspace{-2mm} u_ip^g_{n,k+1}(u_K,u_i-1,u_{J\setminus i})\Big)\\
&=:\frac{u_{k+1}\cdots u_n}{\prod_{i=1}^n u_i!^2}\tilde{p}^g_{n,k}(u_K,u_J)
\end{align*}
Where $\hat{u}_i$ means to exclude the $u_i$ term and we note that $p^g_{n,{k\pm 1}}$ is a polynomial of degree $3g-3+n$, symmetric in the first $k\pm 1$ and last $n-(k\pm 1)$ variables.  Thus $\tilde{p}^g_{n,k}(u_K,u_J)$ has degree $3g-3+n+1$ and the required symmetries.

\end{proof}

\subsubsection{Induction.}  Suppose Proposition~\ref{GWquasi} is true for $g=0$ and $n'<n$.  Apply $$\frac{d^{n-3}}{dt_{b_4}\cdots dt_{b_n}}\Big|_{\bf{t}=0}$$ to (\ref{fullzerotoprec}) and let $\alpha_i=\omega$ to obtain the recursion
\begin{align}\label{zerotoprec2}
\left\langle \tau_{b_1}(\omega)\cdots \tau_{b_n}(\omega) \right\rangle^0&=\sum_{I\subset\{4,\dots,n\}} \Big(\left\langle \tau_0(1)\tau_{b_1-1}(\omega)\tau_I(\omega)\right\rangle^0 \left\langle \tau_0(\omega)\tau_{b_2}(\omega)\tau_{b_3}(\omega)\tau_{CI}(\omega)\right\rangle^0 \\ \nonumber
\nonumber &+\left\langle \tau_0(\omega)\tau_{b_1-1}(\omega)\tau_I(\omega)\right\rangle^0 \left\langle \tau_0(1)\tau_{b_2}(\omega)\tau_{b_3}(\omega)\tau_{CI}(\omega)\right\rangle^0\big)
\end{align}
for $CI=\{b_4,\dots,b_n\}\setminus b_I$.  We now wish to pull out the following factors:  
\begin{align*}\frac{1}{u_i!^2}& \text{ if $b_i=2u_i$, and}\\
\frac{u_i}{u_i!^2}& \text{ if $b_i=2u_i-1$}
\end{align*}to be left with only polynomial terms, of degree up to $n-3$.  By symmetry, we only need to show this for one of the $b_i$, so choose $b_1$.
\paragraph{\emph{Even}} If $b_1=2u_1$, then by induction for $|I|\neq 0,1$ both terms will look like $$\frac{u_1}{u_1!^2}p(u_1)=\frac{1}{u_1!^2}\left[u_1p(u_1)\right]$$ where $u_1p(u_1)$ is a polynomial in $u_1$ of degree $|I|$.

\paragraph{\emph{Odd}} If $b_1=2u_1-1$ then by induction for $|I|\neq 0,1$ both terms will have the form $$\frac{1}{(u_1-1)!^2}p(u_1)=\frac{u_1}{u_1!^2}\left[u_1p(u_1)\right]$$ where $u_1p(u_1)$ is a polynomial in $u_1$ of degree $|I|$.

\paragraph{\emph{Special cases}}We must be careful about the occurrences of one and two point invariants, as the inductive step begins at 3.  These will occur in the first term when $|I|=0$ or $|I|=1$, and the second term when $|I|=0$.  For the first term, application of the string equation leads to $$\langle \tau_{b_1-2}(\omega)\rangle=\begin{cases}\frac{1}{u_1!^2}&b_1=2u_1\\ 0 &b_1=2u_1-1\end{cases}$$ 
or 
\begin{align*}\langle \tau_{b_1-2}(\omega)\tau_{b_i}(\omega)\rangle^0 &+\langle \tau_{b_1-1}(\omega)\tau_{b_i-1}(\omega)\rangle^0\\
& =\begin{cases} \frac{1}{(u_1-1)!^2u_i!^2}\frac{1}{u_1+u_i}+\frac{u_1u_i}{u_1!^2u_i!^2}\frac{1}{u_1+u_i}=\frac{u_1}{u_1!^2u_i!^2} \\   \frac{(u_1-1)u_i}{(u_1-1)!^2u_i!^2}\frac{1}{u_1+u_i-1}+\frac{1}{(u_1-1)!^2(u_i-1)!^2}\frac{1}{u_1+u_i-1}=\frac{u_1u_i}{u_1!^2u_i!^2}u_1\end{cases}
\end{align*}
and we still get the correct form.  If $|I|=0$ the second term is only non zero for $b_1=2u_1-1$ and we get $$\langle \tau_0(\omega)\tau_{2u_1-2}(\omega)\rangle^0 = \frac{1}{(u_1-1)!^2}\frac{1}{u_1}=\frac{u_1}{u_1!^2}$$ which is again correct.

Since $0\leq |I|\leq n-3$, adding terms on the right hand side together gives the required degree of the polynomial part of the stationary Gromov-Witten invariant.

\subsubsection{\bf Genus one case.} {\em Initial case.} This time the induction begins from the one point function.  If we set $\alpha_1=\omega$ and $F=1$ in (\ref{fullonetoprec}) we get
\begin{align*}
\left\langle \tau_{b_1}(\omega) \right\rangle^{1}&=\left\langle \tau_0(1)\tau_{b_1-1}(\omega) \right\rangle^{0}\left\langle \tau_{0}(\omega)\right\rangle^{1}\\\nonumber &+\left\langle \tau_0(\omega)\tau_{b_1-1}(\omega) \right\rangle^{0}\left\langle \tau_{0}(1) \right\rangle^{1}
+\frac{1}{12}\left\langle \tau_0(1)\tau_0(\omega)\tau_{b_1-1}(\omega) \right\rangle^{0}.
\end{align*}
The left hand side is only non zero if $b_1=2u_1$, which makes the second term on the right hand side vanish for dimension reasons.  Using the string equation, the initial terms of the genus zero case and the value \cite{OPaGro}
$$\left\langle \tau_0(\omega)\right\rangle^1=-\frac{1}{24} $$ this reduces to
\begin{equation}\label{GWoneone}
\left\langle \tau_{2u_1}(\omega)\right\rangle^1 =-\frac{1}{24}\frac{1}{u_1!^2}+\frac{1}{12}\frac{1}{(u_1-1)!^2}\frac{1}{u_1}=\frac{1}{24u_1!^2}(2u_1-1).
\end{equation}

\subsubsection{Induction} We have proven the theorem for genus zero and suppose it is true in genus one for $n'<n$.  Let us apply $$\frac{d^{n-1}}{dt_{b_2}\dots dt_{b_{n}}}\Big|_{\bf{t=0}}$$ to (\ref{fullonetoprec}) and let $\alpha_1=\omega, F=1$ to obtain the recursion 
\begin{align}
\left\langle \tau_{b_1}(\omega)\cdots \tau_{b_{n}}(\omega) \right\rangle^1&=\sum_{I\subset\{2,\dots,n\}} \Big(\left\langle \tau_0(1)\tau_{b_1-1}(\omega)\tau_I(\omega)\right\rangle^0 \left\langle \tau_0(\omega)\tau_{CI}(\omega)\right\rangle^1\\
\nonumber &+\left\langle \tau_0(\omega)\tau_{b_1-1}(\omega)\tau_I(\omega)\right\rangle^0 \left\langle \tau_0(1)\tau_{CI}(\omega)\right\rangle^1\big)\\
\nonumber &+\frac{1}{12}\left\langle \tau_0(1)\tau_0(\omega)\tau_{b_1-1}(\omega)\tau_{b_2}(\omega)\cdots \tau_{b_n}(\omega)\right\rangle^0 
\end{align}
for $CI=\{b_2,\dots,b_n\}\setminus b_I$.  As with genus zero, we wish to pull out factors 
\begin{align*}\frac{1}{u_i!^2}& \text{ if $b_i=2u_i$, and}\\
\frac{u_i}{u_i!^2}& \text{ if $b_i=2u_i-1$}
\end{align*}and be left with only polynomial terms, of degree up to $n$.  Again by symmetry we only need to see this for one parameter, so look at $b_1$.
\paragraph{\emph{Even}} For $b_1=2u_1$ the first two terms will be $$\frac{u_1}{u_1!^2}p(u_1)=\frac{1}{u_1!^2}\left[u_1p(u_1)\right]$$ for $u_1p(u_1)$ a polynomial in $u_1$ of degree $|I|$.  The last term will look the same but this time $u_1p(u_1)$ is a polynomial in $u_1$ of degree $n$.
\paragraph{\emph{Odd}} For $b_1=2u_1-1$ the first two terms will be $$\frac{1}{(u_1-1)!^2}p(u_1)=\frac{u_1}{u_1!^2}\left[u_1p(u_1)\right]$$ for $u_1p(u_1)$ a polynomial in $u_1$ of degree $|I|$.  The last term will look the same but this time $u_1p(u_1)$ is a polynomial in $u_1$ of degree $n$.

\paragraph{\emph{Special cases}} We already saw in the genus one proof that application of the string equation to the two point genus zero invariants gave the correct form.

This gives the correct form of all genus one stationary Gromov-Witten invariants, and thus we have proven Proposition~\ref{GWquasi} for $g=0,1$.
\end{proof}
{\em Remark.}  Proposition~\ref{Mpoly} proved a polynomial form (\ref{Mprod}) for the coefficients of an expansion of the Eynard-Orantin invariants $\omega^g_n$ using the transform defined in Lemma~\ref{cortransform}.  The transform is invertible so in particular any power series with coefficients having the polynomial form (\ref{Mprod}) continues analytically to a meromorphic multidifferential on the Riemann surface double covering the plane by $x=z+1/z$.  In particular, Proposition~\ref{GWquasi} proves that the generating functions $\Omega^g_n$ continue analytically to meromorphic multidifferentials over $x=z+1/z$.  This is weaker than Theorem~\ref{th:main} which identifies $\Omega^g_n$ with a known multidifferential.

\subsection{String and dilaton equations for stationary Gromov-Witten invariants}  \label{sec:statgw}

It is easy to see that the divisor equation (\ref{GWdivisor}) restricts to a relationship between purely stationary invariants.  It is subtler that the same is true for the string equation (\ref{GWstring}) and dilaton equation (\ref{GWstring}) which tell us how to remove a non-stationary term and {\em a priori} are not statements about stationary invariants alone.

\begin{proposition}\label{evalatzero}For $g=0$ or $1$, $\tau_0(1)$ classes 
correspond to evaluation of one variable of the stationary invariant polynomial $p^g_{n,k}$ at 0.  More precisely: 
$$\left\langle \tau_0(1)\prod_{i=1}^k\tau_{2u_i}(\omega)\prod_{i=k+1}^{n-1}\tau_{2u_i-1}(\omega)\right\rangle^g = \frac{u_{k+1}...u_{n-1}}{\prod_{i=1}^{n-1} u_i!^2}p^g_{n,k}(u_1,\dots,u_{n-1},0)$$ 
where we have removed from (\ref{eq:quasi}) the factor $u_i/u_i!^2$ corresponding to an odd stationary class and set $u_i=0$.
\end{proposition}
\begin{proof}
We will use induction on $n$ and the topological recursion (\ref{fullzerotoprec}).  

\paragraph{\bf{Genus zero}}Let us begin with the initial cases.  For dimension reasons, we need only check the following two cases whose expressions were computed in section~\ref{sec:proofofmain}.  Interpreting $u_i=0$ to mean ignore the $u_i/u_i!^2$ factor before evaluating gives

\begin{align*} \left\langle \tau_{2u_1-1}(\omega)\tau_{2u_2-1}(\omega)\right\rangle^{0}\Big|_{u_1=0}&=\frac{u_1u_2}{u_1!^2u_2!^2}\frac{1}{u_1+u_2}\Big|_{u_1=0}\\
&=\frac{1}{u_2!^2}= \left\langle \tau_{2u_2-2}(\omega)\right\rangle^{0}\\
&= \left\langle \tau_{0}(1)\tau_{2u_2-1}(\omega)\right\rangle^{0}
\end{align*}
and
\begin{align*} \langle \tau_{2u_1}(\omega)\tau_{2u_2-1}(\omega)\tau_{2u_3-1}(\omega) \rangle^{0}\Big|_{u_3=0}&= \frac{u_2u_3}{u_1!^2u_2!^2u_3!^2}\Big|_{u_3=0}\\
&=\frac{u_2}{u_1!^2u_2!^2}\\
&=\frac{u_1u_2}{u_1!^2u_2!^2}\frac{1}{u_1+u_2}+\frac{1}{u_1!^2(u_2-1)!^2}\frac{1}{u_1+u_2}\\
&= \langle \tau_{2u_1-1}(\omega)\tau_{2u_2-1}(\omega) \rangle^{0}+ \langle \tau_{2u_1}(\omega)\tau_{2u_2-2}(\omega) \rangle^{0}\\
&= \langle \tau_{2u_1}(\omega)\tau_{2u_2-1}(\omega) \tau_0(1)\rangle^{0}
\end{align*}

So that the lemma is true for the smallest cases.  Applying appropriate derivatives to (\ref{fullzerotoprec}) and setting $F=1$ gives the recursion
\begin{align}
\left\langle \tau_{b_1}(\alpha_1)\tau_{b_2}(\alpha_2) \tau_{b_3}(\alpha_3) \tau_{S}(\omega) \right\rangle^0&=\sum_{I\cup J=S} \Big(\left\langle \tau_0(1)\tau_{b_1-1}(\alpha_1)\tau_I(\omega)\right\rangle^0 \left\langle \tau_0(\omega)\tau_{b_2}(\alpha_2)\tau_{b_3}(\alpha_3)\tau_{J}(\omega)\right\rangle^0 \\ \nonumber
\nonumber &+\left\langle \tau_0(\omega)\tau_{b_1-1}(\alpha_1)\tau_I(\omega)\right\rangle^0 \left\langle \tau_0(1)\tau_{b_2}(\alpha_2)\tau_{b_3}(\alpha_3)\tau_{J}(\omega)\right\rangle^0\big)
\end{align}
for $S=\{b_4,\dots,b_n\}$.
We will show by induction that for $\alpha_1=\alpha_2=\alpha_3=\omega$ and $b_2=2u_2-1$ an odd parity variable, the LHS evaluated at $u_2=0$ is equal to the LHS if $\alpha_1=\alpha_3=\omega$, $\alpha_2=1$, $b_2=0$.  The induction will involve equating the right hand sides.
\paragraph{RHS1}Let $\alpha_1=\alpha_3=\omega$, $\alpha_2=1$, $b_2=0$.  Then after applying the divisor equation to the first term and the string equation to the second, the RHS becomes
\begin{align*}
&\sum_{I\cup J=S} \Big(\left\langle \tau_0(1)\tau_{b_1-1}(\omega)\tau_I(\omega)\right\rangle^0 \left\langle \tau_0(1)\tau_{b_3}(\omega)\tau_{J}(\omega)\right\rangle^0\left[ \frac{|J|+b_3+1}{2}\right] \\ \nonumber
\nonumber &+\left\langle \tau_0(\omega)\tau_{b_1-1}(\omega)\tau_I(\omega)\right\rangle^0\big[ \left\langle \tau_{0}(1)\tau_{b_3-1}(\omega)\tau_{J}(\omega)\right\rangle^0+ \left\langle \tau_{0}(1)\tau_{b_3}(\omega)\tau_{J-1}(\omega)\right\rangle^0 \big]\Big)
\end{align*}
where we have used the notation $$\left\langle \tau_{J-1}(\omega)\right\rangle^g=\sum_{b_j\in J}\left\langle \tau_{J\setminus b_j}(\omega)\tau_{b_j-1}(\omega)\right\rangle^g.$$

\paragraph{RHS2} Let  $\alpha_1=\alpha_2=\alpha_3=\omega$ and $b_2=2u_2-1$.  Then applying the divisor equation to the first term and the string equation to the second term, the RHS is 
\begin{align*}
&\sum_{I\cup J=S} \Big(\left\langle \tau_0(1)\tau_{b_1-1}(\omega)\tau_I(\omega)\right\rangle^0 \left\langle \tau_{2u_2-1}(\omega)\tau_{b_3}(\omega)\tau_{J}(\omega)\right\rangle^0\left[  \frac{|J|+b_3+2u_2+1}{2}\right]\\ \nonumber
\nonumber &+\left\langle \tau_0(\omega)\tau_{b_1-1}(\omega)\tau_I(\omega)\right\rangle^0\big[ \left\langle \tau_{2u_2-2}(\omega)\tau_{b_3}(\omega)\tau_{J}(\omega)\right\rangle^0+ \left\langle \tau_{2u_2-1}(\omega)\tau_{b_3-1}(\omega)\tau_{J}(\omega)\right\rangle^0\\
&+ \left\langle \tau_{2u_2-1}(\omega)\tau_{b_3}(\omega)\tau_{J-1}(\omega)\right\rangle^0 \big]\Big)
\end{align*}

By induction, the polynomial expressions for the first and final two terms are equal when we ignore the $\frac{u_2}{u_2!^2}$ factor and put $u_2=0$.  We must look closely at $\left\langle \tau_{2u_2-2}(\omega)\tau_{b_3}(\omega)\tau_{J}(\omega)\right\rangle^0$.  The $u_2$ dependence can be expressed as $$\frac{1}{(u_2-1)!^2}p(u_2)=\frac{u_2}{u_2!^2}\left[ u_2 p(u_2)\right]$$ for $u_2p(u_2)$ a polynomial.  When $u_2$ is set to zero in the polynomial component, this term will vanish and both RHS expressions are equal.

\paragraph{\bf{Genus one}} We shall proceed analogously.  For the smallest case, we may use the genus one topological recursion (\ref{fullonetoprec}), along with the initial computation in section \ref{sec:proofofmain} to find an expression for $\langle \tau_{2u_1-1}(\omega)\tau_{2u_2-1}(\omega)\rangle^1$. Let $b_1=2u_1-1$ and $\alpha_1=\omega$.  Taking a derivative to insert a $\tau_{2u_2-1}(\omega)$ term and discarding parts that are the wrong dimension gives
\begin{align*}
\langle \tau_{2u_1-1}(\omega)&\tau_{2u_2-1}(\omega)\rangle^1=\left\langle \tau_0(1)\tau_{2u_1-2}(\alpha_1)\tau_{2u_2-1}(\omega)\right\rangle^0\left\langle \tau_0(\omega)\right\rangle^1+\left\langle \tau_0(\omega)\tau_{2u_1-2}(\omega)\right\rangle^0\left\langle \tau_0(1)\tau_{2u_2-1}(\omega)\right\rangle^1\\
&\quad\quad\quad\quad\quad\quad+\frac{1}{12}\left\langle\tau_0(1)\tau_0(\omega) \tau_{2u_1-2}(\omega_1)\tau_{2u_2-1}(\omega)\right\rangle^0\\
=&-\frac{1}{24}\big(\left\langle \tau_{2u_1-3}(\alpha_1)\tau_{2u_2-1}(\omega)\right\rangle^0+\left\langle \tau_{2u_1-2}(\alpha_1)\tau_{2u_2-2}(\omega)\right\rangle^0\big)+\left\langle \tau_0(\omega)\tau_{2u_1-2}(\omega)\right\rangle^0\left\langle \tau_{2u_2-2}(\omega)\right\rangle^1\\
&+\frac{1}{12}\left(\left\langle\tau_0(\omega) \tau_{2u_1-3}(\omega_1)\tau_{2u_2-1}(\omega)\right\rangle^0+\left\langle\tau_0(\omega) \tau_{2u_1-2}(\omega_1)\tau_{2u_2-2}(\omega)\right\rangle^0\right)\\
=&-\frac{1}{24}\Big(\frac{(u_1-1)u_2}{(u_1-1)!^2u_2!^2}\frac{1}{u_1+u_2-1}+\frac{1}{(u_1-1)!^2(u_2-1)!^2}\frac{1}{u_1+u_2-1}\Big)\\
&+ \frac{1}{(u_1-1)!^2}\frac{1}{u_1}\frac{2u_2-3}{24(u_2-1)!^2}+\frac{1}{12u_1!^2u_2!^2}\left(u_1^2u_2(u_1-1)+u_1^2u_2^2\right)\\
=&\frac{u_1u_2}{24u_1!^2u_2!^2}(2u_1^2+2u_2^2+2u_1u_2-3u_1-3u_2)
\end{align*}
\normalsize so that
\begin{align*}
\left\langle \tau_{2u_1-1}(\omega)\tau_{2u_2-1}(\omega)\right\rangle^1\Big|_{u_1=0}&=\frac{u_2}{24u_2!^2}(2u_2^2-3u_2)\\
&=\frac{1}{24(u_2-1)!^2}(2u_2-3)=\left\langle \tau_{2u_2-2}(\omega)\right\rangle^1\\
&=\left\langle\tau_0(1) \tau_{2u_2-1}(\omega)\right\rangle^1
\end{align*}
and we have verified the initial case.  Applying appropriate derivatives to (\ref{fullonetoprec}) and setting $\alpha_1=\omega, F=1$ gives the recursion:
\begin{align*}
\langle \tau_{b_1}&(\omega)\tau_{b_2}(\alpha_2)\tau_S(\omega)\rangle^1=\sum_{I\cup J=S}\bigg[ \left\langle \tau_0(1)\tau_{b_1-1}(\omega)\tau_{b_2}(\alpha_2)\tau_{I}(\omega)\right\rangle^0\left\langle \tau_0(\omega)\tau_{J}(\omega)\right\rangle^1 \\
&\hspace{-8mm}+\left\langle \tau_0(\omega)\tau_{b_1-1}(\omega)\tau_{b_2}(\alpha_2)\tau_{I}(\omega)\right\rangle^0\left\langle \tau_0(1)\tau_{J}(\omega)\right\rangle^1+
\left\langle \tau_0(1)\tau_{b_1-1}(\omega)\tau_{I}(\omega)\right\rangle^0\left\langle \tau_0(\omega)\tau_{b_2}(\alpha_2)\tau_{J}(\omega)\right\rangle^1 \\
&\hspace{-2mm}+\left\langle \tau_0(\omega)\tau_{b_1-1}(\omega)\tau_{I}(\omega)\right\rangle^0\left\langle \tau_0(1)\tau_{b_2}(\alpha_2)\tau_{J}(\omega)\right\rangle^1\bigg]
+\frac{1}{12}\left\langle\tau_0(1)\tau_0(\omega) \tau_{b_1-1}(\omega)\tau_{b_2}(\alpha_2)\tau_{S}(\omega)\right\rangle^0
\end{align*}
for $S=\{b_3,\dots,b_n\}$.  Now we may compare expressions.
\paragraph{RHS1}  Let $\alpha_2=1$ and $b_2=0$.  
\begin{align*}
&\sum_{I\cup J=S}\bigg[ \left\langle \tau_0(1)\tau_{0}(1)\tau_{b_1-1}(\omega)\tau_{I}(\omega)\right\rangle^0\hspace{-1mm}\left\langle \tau_0(\omega)\tau_{J}(\omega)\right\rangle^1 +\left\langle\tau_0(\omega) \tau_{0}(1)\tau_{b_1-1}(\omega)\tau_{I}(\omega)\right\rangle^0\hspace{-1mm}\left\langle \tau_0(1)\tau_{J}(\omega)\right\rangle^1\\
&+
\left\langle \tau_0(1)\tau_{b_1-1}(\omega)\tau_{I}(\omega)\right\rangle^0\left\langle \tau_0(\omega)\tau_{0}(1)\tau_{J}(\omega)\right\rangle^1+\left\langle \tau_0(\omega)\tau_{b_1-1}(\omega)\tau_{I}(\omega)\right\rangle^0\left\langle \tau_0(1)\tau_{0}(1)\tau_{J}(\omega)\right\rangle^1\bigg] \\
&
\hspace{+84mm}+\frac{1}{12}\left\langle\tau_0(1)\tau_0(\omega) \tau_{b_1-1}(\omega)\tau_{0}(1)\tau_{S}(\omega)\right\rangle^0.
\end{align*}

\paragraph{RHS2}  Let $\alpha_2=\omega$, $b_2=2u_2-1$.
\begin{align*}
&\sum_{I\cup J=S}\bigg[\left\langle \tau_0(1)\tau_{b_1-1}(\omega)\tau_{2u_2-1}(\omega)\tau_{I}(\omega)\right\rangle^0\left\langle \tau_0(\omega)\tau_{J}(\omega)\right\rangle^1 +\hspace{-0.5mm}\left\langle \tau_0(\omega)\tau_{b_1\hspace{-0.5mm}-\hspace{-0.5mm}1}(\omega)\tau_{2u_2\hspace{-0.5mm}-\hspace{-0.5mm}1}(\omega)\tau_{I}(\omega)\right\rangle^0\hspace{-0.5mm}\left\langle \tau_0(1)\tau_{J}(\omega)\right\rangle^1\\
&\hspace{8mm}+\left\langle \tau_0(1)\tau_{b_1\hspace{-0.5mm}-\hspace{-0.5mm}1}(\omega)\tau_{I}(\omega)\right\rangle^0\hspace{-0.5mm}\left\langle \tau_0(\omega)\tau_{2u_2\hspace{-0.5mm}-\hspace{-0.5mm}1}(\omega)\tau_{J}(\omega)\right\rangle^1 \hspace{-1.5mm}+\left\langle \tau_0(\omega)\tau_{b_1\hspace{-0.5mm}-\hspace{-0.5mm}1}(\omega)\tau_{I}(\omega)\right\rangle^0\left\langle \tau_0(1)\tau_{2u_2\hspace{-0.5mm}-\hspace{-0.5mm}1}(\omega)\tau_{J}(\omega)\right\rangle^1\bigg]\\
&\hspace{+91mm}+\frac{1}{12}\left\langle\tau_0(1)\tau_0(\omega) \tau_{b_1-1}(\omega)\tau_{2u_2-1}(\omega)\tau_{S}(\omega)\right\rangle^0\hspace{-1.5mm}.
\end{align*}

By induction, setting $u_2=0$ in the polynomial expressions for all terms in RHS2 we get equality with RHS1.  (The induction is on $n$, but we have already shown all genus zero to hold.)

\end{proof}

A similar strategy is required for the dilaton equation.  

\begin{proposition}\label{evalatder}For $g=0$ or $1$, $\tau_1(1)$ classes can be evaluated in the expression (\ref{eq:quasi}) by removing the $1/u_i!^2$ factor from an even stationary class and setting $u_i=0$ in the derivative:
\begin{equation}
\left\langle\tau_1(1) \prod_{i=2}^k\tau_{2u_i}(\omega)\prod_{i=k+1}^n\tau_{2u_i-1}(\omega)\right\rangle^g=\frac{u_{k+1}\cdots u_n}{\prod_{i=2}^nu_i!^2}\frac{\partial}{\partial u_1}p^g_{n,k}(u_1,\dots,u_n)\Big|_{u_1=0}
\end{equation}
\end{proposition}
\begin{proof}
We will use induction on $n$ and the topological recursions (\ref{fullzerotoprec}, \ref{fullonetoprec}).

\paragraph{\textbf{Genus zero}}
Begin with the initial cases.   Interpreting operations to mean ignore the $1/u_i!^2$ factor first gives

\begin{align*}\frac{\partial}{\partial u_2} \left\langle \tau_{2u_1}(\omega)\tau_{2u_2}(\omega)\right\rangle^{0}\Big|_{u_2=0}&=\frac{1}{u_1!^2}\frac{\partial}{\partial u_2} \frac{1}{u_1+u_2+1}\Big|_{u_2=0}\\
&=\frac{1}{u_1!^2}\frac{-1}{(u_1+1)^2}= -\left\langle \tau_{2u_1}(\omega)\right\rangle^{0}\\
&= \left\langle \tau_{1}(1)\tau_{2u_1}(\omega)\right\rangle^{0}
\end{align*}
and
\begin{align*}\frac{\partial}{\partial u_3} \langle \tau_{2u_1}(\omega)\tau_{2u_2}(\omega)\tau_{2u_3}(\omega) \rangle^{0}\Big|_{u_3=0}&= \frac{1}{u_1!^2u_2!^2}\frac{\partial}{\partial u_3}1\Big|_{u_3=0}\\
&=0\hspace{+1mm}=\langle\tau_1(1) \tau_{2u_1}(\omega)\tau_{2u_2}(\omega) \rangle^{0}\\
\end{align*}
and
\begin{align*}\frac{\partial}{\partial u_3} \langle \tau_{2u_1-1}(\omega)\tau_{2u_2-1}(\omega)\tau_{2u_3}(\omega) \rangle^{0}\Big|_{u_3=0}&= \frac{u_1u_2}{u_1!^2u_2!^2}\frac{\partial}{\partial u_3}1\Big|_{u_3=0}\\
&=0\hspace{+1mm}=\langle\tau_1(1) \tau_{2u_1-1}(\omega)\tau_{2u_2-1}(\omega) \rangle^{0}\\
\end{align*}
so that the proposition holds for the smallest cases.  Now apply appropriate derivatives to (\ref{fullzerotoprec}) to get the recursion:
\begin{align}
\left\langle \tau_{b_1}(\alpha_1)\tau_{b_2}(\alpha_2) \tau_{b_3}(\alpha_3) \tau_{S}(\omega) \right\rangle^0&=\sum_{I\cup J=S} \Big(\left\langle \tau_0(1)\tau_{b_1-1}(\alpha_1)\tau_I(\omega)\right\rangle^0 \left\langle \tau_0(\omega)\tau_{b_2}(\alpha_2)\tau_{b_3}(\alpha_3)\tau_{J}(\omega)\right\rangle^0 \\ \nonumber
\nonumber &+\left\langle \tau_0(\omega)\tau_{b_1-1}(\alpha_1)\tau_I(\omega)\right\rangle^0 \left\langle \tau_0(1)\tau_{b_2}(\alpha_2)\tau_{b_3}(\alpha_3)\tau_{J}(\omega)\right\rangle^0\big)
\end{align}
for $S=\{b_4,\dots,b_n\}$.  We will show by induction that when $\alpha_1=\alpha_2=\alpha_3=\omega$ and $b_2=2u_2$ an even parity variable, if we ignore the $1/u_2!^2$ factor, take the derivative and set $u_2=0$, the LHS is the same as the LHS when $\alpha_1=\alpha_3=\omega$, $\alpha_2=1$ and $b_2=1$.

\paragraph{RHS1}Let $\alpha_1=\alpha_3=\omega$, $\alpha_2=1$, $b_2=1$.  Then after applying the divisor equation to the first term and the string equation to the second, the RHS becomes
\begin{align*}
&\sum_{I\cup J=S} \Big(\left\langle \tau_0(1)\tau_{b_1-1}(\omega)\tau_I(\omega)\right\rangle^0 \big[\left\langle \tau_0(\omega)\tau_{b_3}(\omega)\tau_{J}(\omega)\right\rangle^0+\left\langle \tau_1(1)\tau_{b_3}(\omega)\tau_{J}(\omega)\right\rangle^0\left[ \frac{|J|+b_3+2}{2}\right] \big]\\ \nonumber
\nonumber &+\left\langle \tau_0(\omega)\tau_{b_1-1}(\omega)\tau_I(\omega)\right\rangle^0\big[ \left\langle \tau_{0}(1)\tau_{b_3}(\omega)\tau_{J}(\omega)\right\rangle^0+\left\langle \tau_{1}(1)\tau_{b_3-1}(\omega)\tau_{J}(\omega)\right\rangle^0+ \left\langle \tau_{1}(1)\tau_{b_3}(\omega)\tau_{J-1}(\omega)\right\rangle^0 \big]\Big)
\end{align*}

\paragraph{RHS2} Let  $\alpha_1=\alpha_2=\alpha_3=\omega$, $b_2=2u_2$.  After applying the divisor equation to the first term and the string equation to the second term, the RHS becomes
\begin{align*}
&\sum_{I\cup J=S} \Big(\left\langle \tau_0(1)\tau_{b_1-1}(\omega)\tau_I(\omega)\right\rangle^0 \left\langle \tau_{2u_2}(\omega)\tau_{b_3}(\omega)\tau_{J}(\omega)\right\rangle^0\left[ \frac{|J|+b_3+2u_2+2}{2}\right] \\ \nonumber
\nonumber &+\left\langle \tau_0(\omega)\tau_{b_1-1}(\omega)\tau_I(\omega)\right\rangle^0\big[ \left\langle \tau_{2u_2-1}(\omega)\tau_{b_3}(\omega)\tau_{J}(\omega)\right\rangle^0+ \left\langle \tau_{2u_2}(\omega)\tau_{b_3-1}(\omega)\tau_{J}(\omega)\right\rangle^0\\&+\left\langle \tau_{2u_2}(\omega)\tau_{b_3}(\omega)\tau_{J-1}(\omega)\right\rangle^0\big]\Big)
\end{align*}
By induction, pulling out $\frac{1}{u_2!^2}$, taking the derivative and setting $u_2=0$ gives equality with the last two terms of each RHS.  For the first term, the product rule on $u_2$ and induction give equality with the first two terms, and all that remains is to check the third term.  We may write the $u_2$ dependence in  $ \left\langle \tau_{2u_2-1}(\omega)\tau_{b_3}(\omega)\tau_{J}(\omega)\right\rangle^0$ as $$\frac{u_2}{u_2!^2}p(u_2)=\frac{1}{u_2!^2}\left[u_2p(u_2)\right]$$ so that when the product rule is used, and $u_2$ subsequently set to zero, this is equivalent to ignoring a $u_2/u_2!^2$ factor and setting $u_2=0$ in the polynomial part.  That is, by Proposition~\ref{evalatzero}, $\left\langle \tau_{0}(1)\tau_{b_3}(\omega)\tau_{J}(\omega)\right\rangle^0$.  Performing these evaluations gives an overall equality.

\paragraph{\textbf{Genus one}} 
Begin with the initial case.   Again interpreting $u_i=0$ to mean ignore the $1/u_i!^2$ factor before performing any operations gives
\begin{align*}
\frac{\partial}{\partial u_2} \left\langle \tau_{2u_1}(\omega)\tau_{2u_2}(\omega)\right\rangle^1\Big|_{u_2=0}&=\frac{1}{u_1!^2}\frac{\partial}{\partial u_2}\frac{1}{24}(2u_1^2+2u_2^2+2u_1u_2-u_1-u_2)\Big|_{u_2=0}\\
&=\frac{1}{24u_1!^2}(2u_1-1)=\left\langle\tau_1(1)\tau_{2u_1}(\omega)\right\rangle^1
\end{align*}
and the proposition holds.  Applying appropriate derivatives to (\ref{fullonetoprec}) and setting $\alpha_1=\omega, F=1$ gives the recursion:

\begin{align*}
\langle \tau_{b_1}&(\omega)\tau_{b_2}(\alpha_2)\tau_S(\omega)\rangle^1=\sum_{I\cup J=S}\bigg[ \left\langle \tau_0(1)\tau_{b_1-1}(\omega)\tau_{b_2}(\alpha_2)\tau_{I}(\omega)\right\rangle^0\left\langle \tau_0(\omega)\tau_{J}(\omega)\right\rangle^1 \\
&\hspace{-8mm}+\left\langle \tau_0(\omega)\tau_{b_1-1}(\omega)\tau_{b_2}(\alpha_2)\tau_{I}(\omega)\right\rangle^0\left\langle \tau_0(1)\tau_{J}(\omega)\right\rangle^1+
\left\langle \tau_0(1)\tau_{b_1-1}(\omega)\tau_{I}(\omega)\right\rangle^0\left\langle \tau_0(\omega)\tau_{b_2}(\alpha_2)\tau_{J}(\omega)\right\rangle^1 \\
&\hspace{-2mm}+\left\langle \tau_0(\omega)\tau_{b_1-1}(\omega)\tau_{I}(\omega)\right\rangle^0\left\langle \tau_0(1)\tau_{b_2}(\alpha_2)\tau_{J}(\omega)\right\rangle^1\bigg]
+\frac{1}{12}\left\langle\tau_0(1)\tau_0(\omega) \tau_{b_1-1}(\omega)\tau_{b_2}(\alpha_2)\tau_{S}(\omega)\right\rangle^0
\end{align*}
for $S=\{b_3,\dots,b_n\}$.  Now we may compare expressions.

\paragraph{RHS1} Let $\alpha_2=1$ and $b_2=1$.  
\begin{align*}
&\sum_{I\cup J=S}\bigg[ \left\langle \tau_0(1)\tau_{1}(1)\tau_{b_1-1}(\omega)\tau_{I}(\omega)\right\rangle^0\left\langle \tau_0(\omega)\tau_{J}(\omega)\right\rangle^1+\left\langle\tau_0(\omega) \tau_{1}(1)\tau_{b_1-1}(\omega)\tau_{I}(\omega)\right\rangle^0\left\langle \tau_0(1)\tau_{J}(\omega)\right\rangle^1\\
&\hspace{+4mm}+\left\langle \tau_0(1)\tau_{b_1-1}(\omega)\tau_{I}(\omega)\right\rangle^0\left\langle \tau_0(\omega)\tau_{1}(1)\tau_{J}(\omega)\right\rangle^1+\left\langle \tau_0(\omega)\tau_{b_1-1}(\omega)\tau_{I}(\omega)\right\rangle^0\left\langle \tau_0(1)\tau_{1}(1)\tau_{J}(\omega)\right\rangle^1\bigg] \\
&
\hspace{+84mm}+\frac{1}{12}\left\langle\tau_0(1)\tau_0(\omega) \tau_{b_1-1}(\omega)\tau_{1}(1)\tau_{S}(\omega)\right\rangle^0.
\end{align*}

\paragraph{RHS2} Let $\alpha_2=\omega$, $b_2=2u_2$.
\begin{align*}
&\sum_{I\cup J=S}\bigg[\left\langle \tau_0(1)\tau_{b_1-1}(\omega)\tau_{2u_2}(\omega)\tau_{I}(\omega)\right\rangle^0\left\langle \tau_0(\omega)\tau_{J}(\omega)\right\rangle^1 +\left\langle \tau_0(\omega)\tau_{b_1-1}(\omega)\tau_{2u_2}(\omega)\tau_{I}(\omega)\right\rangle^0\hspace{-0.5mm}\left\langle \tau_0(1)\tau_{J}(\omega)\right\rangle^1\\
&\hspace{+6mm}+\left\langle \tau_0(1)\tau_{b_1-1}(\omega)\tau_{I}(\omega)\right\rangle^0\hspace{-0.5mm}\left\langle \tau_0(\omega)\tau_{2u_2}(\omega)\tau_{J}(\omega)\right\rangle^1 +\left\langle \tau_0(\omega)\tau_{b_1-1}(\omega)\tau_{I}(\omega)\right\rangle^0\left\langle \tau_0(1)\tau_{2u_2}(\omega)\tau_{J}(\omega)\right\rangle^1\bigg]\\
&\hspace{+88mm}+\frac{1}{12}\left\langle\tau_0(1)\tau_0(\omega) \tau_{b_1-1}(\omega)\tau_{2u_2}(\omega)\tau_{S}(\omega)\right\rangle^0\hspace{-1.5mm}.
\end{align*}
Given the proposition is true in genus zero, we need only consider terms three and four. By induction on $n$, when we ignore the $1/u_2!^2$ factor, take the derivative and evaluate at $u_2=0$ the polynomial expressions for RHS2, we get equality with RHS1.

\end{proof}
{\em Remarks.}  1. Combining Proposition~\ref{evalatzero}, respectively Proposition~\ref{evalatder}, with the string equation, respectively the dilaton equation, gives relations between stationary invariants alone.  One might call these string and dilaton equations for stationary invariants.

2. We expect Propositions~\ref{GWquasi}, \ref{evalatzero} and \ref{evalatder} to hold for the Gromov-Witten invariants of $\bp^1$ for all genus $g$.\\

\begin{theorem}\label{GWgenus0}The divisor and string equations uniquely determine all genus zero and one stationary Gromov-Witten invariants.
\end{theorem}

\begin{proof}
We begin with the genus zero case and use the $g=0$ form of Proposition~\ref{GWquasi}:
\begin{equation*}
\left\langle \prod_{i=1}^k\tau_{2u_i}(\omega)\prod_{i=k+1}^n\tau_{2u_i-1}(\omega)\right\rangle^0=\frac{u_{k+1}\cdots u_n}{\prod_{i=1}^n u_i!^2}p^0_{n,k}(u_1,\dots,u_n)
\end{equation*}
where $p^0_{n,k}(u_1,\dots,u_n)$ is a polynomial of degree $n-3$ in the $u_i$'s, symmetric in the first $k$ and the last $n-k$ variables.  

The divisor equation enables one to compute $p^g_{n,k}(0,u_2,\dots,u_n)$ from $p^g_{n-1,k-1}(u_2,\dots,u_n)$.  By symmetry, this equates evaluation of any of the first $k$ variables  at 0 to known functions.  Proposition~\ref{evalatzero} and the string equation enable one to compute $p^g_{n,k}(u_1,\dots,u_{n-1},0)$ from $p^g_{n-1,k}(u_1,\dots,u_{n-1})$ which by symmetry gives evaluation of any of the last $n-k$ variables.  Thus we can apply Lemma~\ref{th:sympol} to deduce the genus zero case.

The genus one case relies on the $g=0,1$ version of Proposition~\ref{GWquasi} and the $g=0$ version of Theorem~\ref{th:main} which requires only the $g=0$ of Theorem~\ref{GWgenus0} proven above.  Proposition~\ref{evalatzero} and Lemma~\ref{th:sympol} prove that the string and divisor equations determine the $p^1_{n,k}$'s up to a constant.  (This time $H$ has degree $n$ so that $\tilde{H}$ is a constant.)  The constant is the coefficient of $u_1\cdots u_n$.  We use the genus one topological recursion and theorem \ref{th:main} for genus zero to determine this coefficient.  Having taken the appropriate derivatives the recursion (\ref{fullonetoprec}) becomes
\begin{align}
\left\langle \tau_{b_1}(\omega)\cdots \tau_{b_{n}}(\omega) \right\rangle^1&=\sum_{I\subset\{2,\dots,n\}} \Big(\left\langle \tau_0(1)\tau_{b_1-1}(\omega)\tau_I(\omega)\right\rangle^0 \left\langle \tau_0(\omega)\tau_{CI}(\omega)\right\rangle^1\\
\nonumber &+\left\langle \tau_0(\omega)\tau_{b_1-1}(\omega)\tau_I(\omega)\right\rangle^0 \left\langle \tau_0(1)\tau_{CI}(\omega)\right\rangle^1\big)\\
\nonumber &+\frac{1}{12}\left\langle \tau_0(1)\tau_0(\omega)\tau_{b_1-1}(\omega)\tau_{b_2}(\omega)\cdots \tau_{b_n}(\omega)\right\rangle^0 
\end{align}
for $CI=\{b_2,\dots,b_n\}\setminus b_I$.
For each $i=1,\dots,n$ let $b_i=2u_i$ or $2u_i-1$ depending on parity.  Proposition~\ref{GWquasi} shows that removing the appropriate binomial coefficients\footnote{We must be careful of the parity of the $b_1$ terms, along with the special cases $I=\phi, \{2,\dots , n\}$ as in section \ref{sec:gwpol}.} leaves the first two terms as polynomials in the $u_i$'s of degree $n-1$.  The third term leaves a polynomial of degree $n$, which by Proposition~\ref{evalatzero} is:
$$u_1p^0_{n+2,k}(0,u_1,\dots,u_n,0)$$ where we have evaluated one of the first $k$ and one of the last $n-k$ variables at zero.  

Note that we write the first $k$ variables corresponding to the even parity $b_i$'s and so as written above we have assumed $b_1$ to be even.  Shuffling the parameters in $p^0_{n+2,k}$ will give the argument for $b_1$ odd.

The monomial $u_2\cdots u_n$ appears in $p^0_{n+2,k}(0,u_1,\dots,u_n,0)$ as a top degree term and using the genus zero equality with the Eynard-Orantin expansion (\ref{g0mp}): $$p^0_{n+2,k}(0,u_1,\dots,u_n,0)=m^0_{n+2,k}(1,2u_1+1,\dots,2u_n,0),$$ Proposition~\ref{Mpoly} computes this coefficient to be $$\frac{2}{2^{n}}\left\langle \tau_1^{n-1}\tau_0^3\right\rangle 2^{n-1}=(n-1)!$$ The extra factors of 2 come from the change of variables $b=2u$ or $2u+1$.  Thus the coefficient of $u_1\cdots u_n$ in $p^1_{n,k}(u_1,\dots,u_n)$ is $$\frac{1}{12}(n-1)!$$
\end{proof}

We are finally in a position to prove Theorem~\ref{th:main}.
\begin{proof}[Proof of Theorem~\ref{th:main}.]

{\em Genus zero.} Begin with the divisor (\ref{Mdivisor}) and string (\ref{Mstring}) equations for the $M^g_n$'s.  If we divide by $\prod_{i=1}^n b_1!$ and shift the arguments by one, we obtain the form of the string and divisor equations ((\ref{GWstring}), (\ref{GWdivisor})) for the stationary Gromov-Witten invariants.  Recall that the string equation for the stationary Gromov-Witten invariants involves evaluation at -1 by combining Proposition~\ref{evalatzero} with the string equation.

Since both invariants are uniquely determined by these same equations (Theorems \ref{GWgenus0} and \ref{th:Mgenus0}), all we need to do is check that the initial cases match.  We can explicitly calculate the $(g,n)=(0,3)$ case: using the topological recursion for Gromov-Witten invariants we have already seen (\ref{03GW}):
\begin{align*}
\langle \tau_{2u_1}(\omega)\tau_{2u_2}(\omega)\tau_{2u_3}(\omega) \rangle^{0} &= \frac{1}{u_1!^2u_2!^2u_3!^2}\\
 \langle \tau_{2u_1}(\omega)\tau_{2u_2-1}(\omega)\tau_{2u_3-1}(\omega) \rangle^{0}&= \frac{u_2u_3}{u_1!^2u_2!^2u_3!^2}.
\end{align*}
Using \cite{NScPol}, we can compute the expansion of the $(g,n)=(0,3)$ Eynard-Orantin invariants for the curve (\ref{eq:lnz}):
\begin{align*}
M^0_3(2u_1+1,2u_2+1,2u_3+1) &=\prod_{i=1}^3(2u_i+1) \binom{2u_i}{u_i}=\prod_{i=1}^3(2u_i+1)!\langle \tau_{2u_1}(\omega)\tau_{2u_2}(\omega)\tau_{2u_3}(\omega) \rangle^{0}\\
M^0_3(2u_1+1,2u_2,2u_3)&= (2u_1+1)u_2u_3\prod_{i=1}^3\binom{2u_i}{u_i}\\&=(2u_1+1)!(2u_2)!(2u_3)!\langle \tau_{2u_1}(\omega)\tau_{2u_2-1}(\omega)\tau_{2u_3-1}(\omega) \rangle^{0}
\end{align*}
and so the theorem is true in genus zero and we have the equality 
\begin{equation}\label{g0mp}
p^0_{n,k}(u_1,\dots,u_n)=m^0_{n,k}(2u_1+1,\dots,2u_k+1,2u_{k+1},\dots,2u_n).\end{equation}
The 1-point and 2-point cases given in (\ref{eq:excep})

{\em Genus one.} This time both sets of invariants are determined by the string and divisor equations and the coefficient of the top degree polynomial terms.  We must check that the initial cases and that the top coefficients agree.  We already saw the initial Gromov-Witten invariant (\ref{GWoneone}):
\begin{equation*}
\left\langle \tau_{2u}(\omega)\right\rangle^1 =\frac{1}{24u!^2}(2u-1).
\end{equation*}
Using \cite{NScPol} and lemma \ref{cortransform} we can compute the expansion of the $(g,n)=(1,1)$ Eynard-Orantin invariant:
\begin{equation*}
M^1_1(2u_1+1)=(2u_1+1)\binom{2u_1}{u_1}\frac{1}{24}(2u_1-1)=(2u_1+1)!\left\langle \tau_{2u_1}(\omega)\right\rangle^1
\end{equation*}
which are the required Gromov-Witten invariants, so the initial cases hold.  Furthermore, the coefficient of $u_1\cdots u_n$ in $m^1_{n,k}(2u_1+1,\dots 2u_k+1,2u_{k+1},\dots,2u_n)$ is $\frac{1}{12}(n-1)!$ by Proposition~\ref{Mpoly}.  Thus the theorem is true in genus one.
\end{proof}
{\em Remark.} 
The identification of the coefficients $M^g_n$ in the expansion of $\omega^g_n$ around $x_i=\infty$ with Gromov-Witten invariants raises the question of finding a similar geometric interpretation of $N^g_n$ %(= coefficients in the expansion of $\omega^g_n$ in the coordinates $z_i$) 
which is related to $M^g_n$ via Lemma~\ref{ntom}.  The $N^g_n$ are much simpler and contain the essential information of the $M^g_n$ and hence the Gromov-Witten invariants.

A corollary of Theorem~\ref{th:main} is Theorem~\ref{th:GWquasi}.
\begin{proof}[Proof of Theorem~\ref{th:GWquasi}]
For $g=0,1$, Theorem~\ref{th:main} allows us to identify 
\[M^g_{n,k}=\prod_{i=1}^n(b_i+1)!\left\langle \prod_{i=1}^k\tau_{2u_i}(\omega)\prod_{i=k+1}^n\tau_{2u_i-1}(\omega)\right\rangle^g\] 
under the substitution $b_i=2u_i(+1)$, and hence their polynomial parts 
$m^g_{n,k}(b_1,...,b_n)=p^g_{n,k}(u_1,...,u_n)$
defined in Propositions~\ref{Mpoly} and \ref{GWquasi}.

Proposition~\ref{Mpoly} gives the coefficient of $b_1^{\beta_1}...b_n^{\beta_n}=b^{\beta}$ in $m^g_{n,k}$ as $v_{\beta}=0$ or $v_{\beta}= 2^{-2g+3-n}\int_{\overline{\modm}_{g,n}}\psi_1^{\beta_1}...\psi_n^{\beta_n}$ since $y'(1)=1$ and $y'(-1)=-1$ in (\ref{Mcoeff}).
 
Hence the top coefficients $c_{\beta}$ of $u_1^{\beta_1}\cdots u_n^{\beta_n}$ which satisfy $c_{\beta}=v_{\beta}\cdot 2^{3g-3+n}$ are given by 
\begin{equation}
c_{\beta}=2^g\int_{\overline{\modm}_{g,n}}\psi_1^{\beta_1}...\psi_n^{\beta_n}\end{equation}
for $|\beta|=3g-3+n$. 
\end{proof}

\section{Virasoro constraints}  \label{sec:vir}

The Gromov-Witten invariants of $\bp^1$ satisfy the following recursions for each $k>0$  known as Virasoro constraints:
\begin{align*}
(k+&1)!\langle[\tau_{k+1}(1)+2c_{k+1}\tau_k(\omega)]\tau_{b_S}(\omega)\rangle^g
-\sum_{j=1}^n\frac{(k+b_j+1)!}{b_j!}\langle\tau_{k+b_j}(\omega)\tau_{b_1}(\omega)..\widehat{\tau}_{b_j}(\omega)..\tau_{b_n}(\omega)\rangle^g\\
&= \sum_{m=0}^{k-2}(m+1)!(k-m-1)!\biggr[\langle\tau_{m}(\omega)\tau_{k-m-2}(\omega)\tau_{b_S}(\omega)\rangle^{g-1}+\hspace{-5mm}\displaystyle\sum_{\begin{array}{c}_{g_1+g_2=g}\\_{I\sqcup J=S}\end{array}}\hspace{-5mm}
\langle\tau_m(\omega)\tau_{b_I}(\omega)\rangle^{g_1} \langle\tau_{k-m-2}(\omega)\tau_{b_J}(\omega)\rangle^{g_2}\biggr]
\end{align*}
for $c_k=1+1/2+...+1/k$ and $\tau_{b_K}(\omega)=\prod_{j\in K} \tau_{b_j}(\omega)$.

In terms of the generating functions, the Virasoro constraints become:
 
\begin{equation}  \label{eq:virgen}
\eta^g_{n+1}(x,x_S)=\Omega^{g-1}_{n+2}(x,x,x_S)+\hspace{-5mm}\displaystyle\sum_{\begin{array}{c}_{g_1+g_2=g}\\_{I\sqcup J=S}\end{array}}\hspace{-5mm}
\Omega^{g_1}_{|I|+1}(x,x_I)\Omega^{g_2}_{|J|+1}(x,x_J)-\sum_{i=1}^n\frac{dxdx_i}{(x-x_i)^2}\Omega^g_n(x,x_{S\backslash i})
\end{equation}
where non-stationary invariants are stored in the generating function
\begin{align*}
\eta^g_{n+1}(x,x_S)dx^2:=\sum_{b_i\geq 0}\prod_{j=1}^n(b_j+1)!x_j^{-b_j-2}dx_jdx^2\bigg[\sum_{k=0}^{\infty}& (k+1)!x^{-k-2}
\langle[\tau_{k+1}(1)+2c_{k+1}\tau_{k}(\omega)]\tau_{b_S}(\omega)\rangle^g\\
& -\sum_{j=1}^n\frac{1}{b_j!}\sum_{k=0}^{b_j-1}x^{b_j-k-2}\langle\tau_k(\omega)\tau_{b_1}(\omega)..\widehat{\tau}_{b_j}(\omega)..\tau_{b_n}(\omega)\rangle^g\bigg]
\end{align*}
%-\sum_{j=1}^n\frac{(b_0+b_j+1)!}{b_j!(b_0+1)!}\langle\tau_{b_0+b_j}(\omega)\tau_{b_1}(\omega)..\widehat{\tau}_{b_j}(\omega)..\tau_{b_n}(\omega)\rangle^g

A consequence of (\ref{eq:EOrec}) is the following set of {\em loop equations}, also known as Virasoro constraints, satisfied by the Eynard-Orantin invariants \cite{EOrTop}.  The loop equations express the fact that the sum over the fibres of $x$ of a combination of the Eynard-Orantin invariants cancels the poles at the branch points of $x$.  Explicitly, the following function $P^g_{n+1}(x,z_S)$ has no poles at the branch points of $x$:
\begin{align}  \label{eq:loop}
P^g_{n+1}(x,z_S)dx(z)^2&=\sum_{x(z)=x}\biggr[\omega^{g-1}_{n+2}(z,z,z_S)+\hspace{-5mm}\displaystyle\sum_{\begin{array}{c}_{g_1+g_2=g}\\_{I\sqcup J=S}\end{array}}\hspace{-5mm}
\omega^{g_1}_{|I|+1}(z,z_I)\omega^{g_2}_{|J|+1}(z,z_J)\bigg]
\end{align}
or equivalently the right hand side vanishes to order two at each branch point of $x$.  The sum of differentials over fibres of $x$ is to be understood via a common trivialisation of the cotangent bundle supplied by $dx$.  The statement of the loop equations is unchanged if we replace $y(z)$ by $y_N(z)$ for $N\geq 6g-4+2n$.  This is because each $\omega^{g'}_{n'}$ in the equation stabilises in this range, except for $\omega^0_1(z)$.  If $y_N(z)\mapsto y_N(z)+a(1-z^2)^{N+1}$ then $\omega^0_1(z)\omega^g_{n+1}(z,z_S)\mapsto \omega^0_1(z)\omega^g_{n+1}(z,z_S)+(1-z^2)^2h(z)$ for $h$ analytic at $z=\pm1 $ since $a(1-z^2)^{N+1}$ cancels the poles of $\omega^g_{n+1}$.  Hence $P^g_{n+1}(x,z_S)dx(z)^2\mapsto P^g_{n+1}(x,z_S)dx(z)^2+z^2h(z)dx(z)^2$ which still has no poles at $z=\pm 1$.  The proof of (\ref{eq:loop}) uses the fact that the recursion (\ref{eq:EOrec}) is retrieved from
\[0=\sum_{\alpha}\hspace{-2mm}\res{z=\alpha}K(z_0,z)\cdot P^g_{n+1}(x,z_S)dx(z)^2\]
together with the identity $\sum_{x(z)=x}\omega^g_n(z,z_S)=0$ (which has the effect of converting some $z$ to $\hat{z}$.)

For $x=z+1/z$, the involution that swaps branches is given by $\hat{z}=1/z$ and 
\[ \omega^g_n(1/z,z_1)=-\omega^g_n(z,z_1)+
\delta_{g,0}\delta_{n,2}\Omega^0_2(x,x_1).\]  
In particular 
\begin{align*}
P^g_{n+1}(x,z_S)dx(z)^2&=2\biggr[\omega^{g-1}_{n+2}(z,z,z_S)+\hspace{-5mm}\displaystyle\sum_{\begin{array}{c}_{g_1+g_2=g}\\_{I\sqcup J=S}\end{array}}\hspace{-5mm}\omega^{g_1}_{|I|+1}(z,z_I)\omega^{g_2}_{|J|+1}(z,z_J)-\sum_{i=1}^n\frac{dxdx_i}{(x-x_i)^2}\omega^g_n(z,z_{S\backslash i})\bigg]\\
&=2\biggr[\Omega^{g-1}_{n+2}(x,x,x_S)+\hspace{-5mm}\displaystyle\sum_{\begin{array}{c}_{g_1+g_2=g}\\_{I\sqcup J=S}\end{array}}\hspace{-5mm}\Omega^{g_1}_{|I|+1}(x,x_I)\Omega^{g_2}_{|J|+1}(x,x_J)+\sum_{i=1}^n\frac{dxdx_i}{(x-x_i)^2}\Omega^g_n(x,x_{S\backslash i})\bigg].
\end{align*}
Thus the Virasoro constraints and the loop equations agree if
$\eta^g_{n+1}(x,x_S)$ has no poles at the branch points of $x$.
The Virasoro constraints enable one to calculate non-stationary invariants from stationary invariants but does not determine the stationary invariants \cite{OPaVir}.  The Eynard-Orantin recursions determine the stationary invariants by assembling into a generating function maps of all degrees.  This is necessary to make sense of residues away from the point of expansion.  

\section{A Matrix integral proof of Theorem \ref{th:main} for all genus}  \label{sec:mat}

The Eynard-Orantin invariants come from matrix integrals.  In good cases, the expansion of the invariants $\omega^g_n$ around $\{\x_i=\infty\}$ coincides with the expectation value with respect to a measure on the space of Hermitian matrices of the product of resolvents
\[ W^g_n(x_1,\dots,x_n):=\Big\langle \prod_{i=1}^n{\rm Tr}\frac{1}{x_i-M}\Big\rangle^g_{conn}.\]
The right hand side denotes the connected genus $g$ part of the perturbative expansion of the integral which is expanded over a set of fatgraphs that naturally have genus.  The space of matrices may be a variant of the space of Hermitian matrices.

\textbf{Plancherel Measure.}
There is a natural measure on partitions given by the Plancherel measure, using the dimension of irreducible representations of $S_N$, labeled by partitions $\lambda$ and satisfying $\sum_{|\lambda|=N}\dim(\lambda)^2=N!$.  We can use Eynard-Orantin techniques to study expectation values of the partition function 
\[Z_N(Q)=\sum_{l(\lambda)\leq N}\left(\frac{\dim\lambda}{|\lambda|!}\right)^2Q^{2|\lambda|}.\]
The asymptotic expansion of $Z_N$ as $Q\rightarrow \infty$
\[\ln Z_N(Q)=\sum_gQ^{2-2g}F^g\]
can be solved using the normalisation of the Plancherel measure.  For $N\to\infty$, $\exp(-Q^2)Z_N(Q)\to 1$ so 
\[F^g=\delta_{g,0}.\]

Expectation values of $Z_N$ can be generated by the spectral curve \cite{EynOrd}
\begin{equation*} 
C=\begin{cases}x=z+1/z\\ 
y=\ln{z}
\end{cases}
\end{equation*}
In particular, we will prove that if $M^g_n(b_1,\dots,b_n)$ are the coefficients of the Eynard-Orantin invariant $\omega^g_n$ in the expansion about $x=\infty$, then the $M^g_n$'s can be expressed as stationary Gromov-Witten invariants. 
\begin{proof}[Proof of theorem \ref{th:main}]
We use the expression of Okounkov and Pandharipande (\ref{eq:gwplanch}) that relates Gromov-Witten invariants to the Plancherel measure:
\begin{equation*}
\big\langle \prod_{i=1}^n \tau_{b_i}(\omega)\big\rangle_d^{\bullet } = \sum_{|\lambda |=d} \Big(\frac{\dim\lambda}{d!}\Big)^2\prod_{i=1}^n \frac{\textbf{p}_{b_i+1}(\lambda)}{(b_i+1)!}
\end{equation*}
for
$$\textbf{p} _{k}(\lambda)=\sum_{i=1}^\infty \big[(\lambda_i-i+\frac{1}{2})^{k}-(-i+\frac{1}{2})^k\big]+(1-2^{-k})\zeta(-k).$$
In \cite{EynOrd}, it is shown that the Plancherel measure can be written in the large $N$ limit as a matrix integral:
\begin{equation}
\sum_{l(\lambda)\leq N} \Big(\frac{\dim\lambda}{|\lambda|!}\Big)^2Q^{2|\lambda|}=\frac{Q^{N^2}}{N!}\int_{H_N(\mathcal{C})}e^{-Qtr(V(X))}dX
\end{equation}
where $QV(x)=\ln(\Gamma(Qx))-\ln(\Gamma(-Qx))+i\pi Qx+\ln(Qx)-Qx\ln Q+QA_0$ for some constant $A_0$, $\mathcal{C}$ is a contour in the complex plane surrounding all of the positive integers and $H_N(\mathcal{C})$ is the set of normal $N\times N$ matrices whose eigenvalues lie on the contour $\mathcal{C}$.
\[H_N(\mathcal{C})=\left\{X\Big| X=U^T\Lambda U, \quad UU^T\hspace{-1.5mm}=Id_N,\quad  \Lambda=\text{diag}(\lambda_1,\dots,\lambda_N),\quad \lambda_i\in \mathcal{C} \right\}.\]  It was also found that this matrix model has a rational spectral curve given by
\begin{equation} 
\tilde{C}=\begin{cases}x=\frac{N-1/2}{Q}+z+1/z\\ 
y=\text{ln}(z).
\end{cases}
\end{equation} Thus the $\tilde{M}^g_n$'s of $\tilde{C}$ correspond to expectation values in this integral, or equivalently expectation values of the Plancherel measure.  If $h_i$ represent the $\pi/4$ rotated partitions, $h_i=\lambda_i-i+N$, then

\begin{align*}
W^g_n(x_1,\dots,x_n)&:=\Big\langle \prod_{i=1}^n\sum_j \frac{1}{x_i-h_j/Q}\Big\rangle^g_{conn}\\
&=\sum_{b_1,\dots,b_n=0}^\infty \frac{\tilde{M}^g_n(b_1,\dots,b_n)}{x_1^{b_1+1}\cdots x_n^{b_n+1}}\\
&=\sum_{b_1,\dots,b_n=0}^\infty\frac{1}{x_1^{b_1+1}\cdots x_n^{b_n+1}}\Big[\sum_{l(\lambda)\leq N} \Big(\frac{\dim\lambda}{|\lambda|!}\Big)^2Q^{2|\lambda|-\sum b_i}\prod_{i=1}^n \sum_j h_j^{k_i}\Big]^g_{conn}\\
&=\sum_{b_1,\dots,b_n=0}^\infty\frac{1}{x_1^{b_1+1}\cdots x_n^{b_n+1}}\Big[\sum_{l(\lambda)\leq N} \Big(\frac{\dim\lambda}{|\lambda|!}\Big)^2Q^{2|\lambda|-\sum b_i}\prod_{i=1}^n \sum_j(\lambda_j-j+N)^{b_i}\Big]^g_{conn}.\\
\end{align*}
Since Eynard-Orantin invariants don't change when $x$ changes by a constant, we can consider the curve 
\begin{equation} 
C_2=\begin{cases}x'=z+1/z\\ 
y=\text{ln}(z).
\end{cases}
\end{equation}
The $\omega^g_n$ will be the same, but the expansion around $x'=\infty$ will be different, and we get new $M^g_n$'s:
\begin{align*}
W^g_n(x_1,\dots,x_n)&=\Big\langle \prod_{i=1}^n\sum_j \frac{1}{x_i'+(N-1/2)/Q-h_j/Q}\Big\rangle^g_{conn}\\
&=\sum_{b_1,\dots,b_n=0}^\infty \frac{M^g_n(b_1,\dots,b_n)}{x_1'^{b_1+1}\cdots x_n'^{b_n+1}}\\
%&=\sum_{k_1,\dots,k_n=0}^\infty\frac{1}{x_1'^{k_1+1}\cdots x_n'^{k_n+1}}\sum_{l(\lambda)\leq N} (\frac{\text{dim}(\lambda)}{|\lambda|})Q^{2|\lambda|-k_1-\cdots-k_n}\prod_{i=1}^n \sum_j (h_j-N+\frac{1}{2})^{k_i}\\
%&=\sum_{k_1,\dots,k_n=0}^\infty\frac{1}{x_1'^{k_1+1}\cdots x_n'^{k_n+1}}\sum_{l(\lambda)\leq N} (\frac{\text{dim}(\lambda)}{|\lambda|})Q^{2|\lambda|-k_1-\cdots-k_n}\prod_{i=1}^n \sum_j(\lambda_j-j+\frac{1}{2})^{k_i}\\
%&=\sum_{k_1,\dots,k_n=0}^\infty\frac{1}{x_1'^{k_1+1}\cdots x_n'^{k_n+1}}\sum_{l(\lambda)\leq N} (\frac{\text{dim}(\lambda)}{|\lambda|})Q^{2|\lambda|-k_1-\cdots-k_n}\prod_{i=1}^n \textbf{p} _{k}(\lambda) - \sum_j(-j+\frac{1}{2})^{k_i}-(1-2^{k_i})\zeta(-k_i)\\
%&=\sum_{k_1,\dots,k_n=0}^\infty\frac{1}{x_1'^{k_1+1}\cdots x_n'^{k_n+1}}\sum_{l(\lambda)\leq N} (\frac{\text{dim}(\lambda)}{|\lambda|})Q^{2|\lambda|-k_1-\cdots-k_n}\prod_{i=1}^n \textbf{p} _{k}(\lambda) 
\end{align*}

where
\begin{align*}
M^g_n(b_1,\dots,b_n)&=\Big[\sum_{l(\lambda)\leq N} \Big(\frac{\dim\lambda}{|\lambda|!}\Big)^2Q^{2|\lambda|-b_1-\cdots-b_n}\prod_{i=1}^n \sum_j (h_j-N+\frac{1}{2})^{b_i}\Big]^g_{conn}\\
&=\Big[\sum_dQ^{2d-\sum b_i}\sum_{|\lambda|=d} \Big(\frac{\dim\lambda}{|\lambda|!}\Big)^2\prod_{i=1}^n \sum_j(\lambda_j-j+\frac{1}{2})^{b_i}\Big]^g_{conn}\\
&=\Big[\sum_dQ^{2d-\sum b_i}\sum_{|\lambda|=d} \Big(\frac{\dim\lambda}{|\lambda|!}\Big)^2\prod_{i=1}^n (\textbf{p} _{b_i}(\lambda) + \sum_j(-j+\frac{1}{2})^{b_i}-(1-2^{b_i})\zeta(-b_i))\Big]^g_{conn}\\
&=\Big[\sum_dQ^{2d-\sum b_i}\prod_{i=1}^n b_i!\big\langle \prod_{i=1}^n \tau_{b_i-1}(\omega)\big\rangle_d^{\bullet }\Big]^g_{conn}+0.
\end{align*}
Using the fact that $$\sum_{j=1}^\infty \sum_{k=0}^\infty \frac{(-j+\frac{1}{2})^kz^k}{k!}=\sum_{j=1}^\infty e^{z(-j+\frac{1}{2})}=e^{\frac{z}{2}}(\frac{1}{1-e^{-z}}-1)=\frac{2}{\text{sinh}(z/2)}=\sum_{k=0}^\infty \frac{(1-2^{-k})\zeta(-k)}{k!}z^{k}$$ and comparing coefficients.  Note that these extra components of $\textbf{p}_{k}$ are only used in \cite{OPaGro} so that evaluations can be made for finite partitions without the need to evaluate infinite series.  In an expectation value they will have no effect.  Since $2g-2+2d=\sum_{i=1}^n(b_i-1)$ defines the degree, taking the genus $g$ component involves taking only one term, and we extract the coefficient of $Q^{2-2g-n}$.  The connected part then gives connected Gromov-Witten invariants:
\begin{equation*}
M^g_n(b_1,\dots,b_n)=\prod_{i=1}^n b_i!\big\langle \prod_{i=1}^n \tau_{b_i-1}(\omega)\big\rangle^{g}.
\end{equation*}
\end{proof}

%\emph{Remark.} It is interesting that in the discrete measure case, the shifted symmetric functions have replaced an insertion of $tr(X^b)$ inside the continuous matrix integral.

\section{Formulae} \label{sec:formulae}

The following values for $N^g_{n,k}$ were computed with the method of \cite{NScPol} and using lemma \ref{cortransform} we can compute the corresponding $m^g_{n,k}$'s.
\small
%\begin{table}[h!]  \label{tab:poly3}
%\caption{$y=\ln{z}$}
\begin{center}%{$y=\ln{z}$}
\begin{spacing}{1.6}  
\begin{tabular}{||l|c|c|c|c||} 
\hline\hline

{\bf g} &{\bf n}&{\bf k}&$\bf N^g_{n,k}(b_1,...,b_n)$&$\bf m^g_{n,k}(b_1,\dots,b_n)$\\ \hline

0&3&0,2&$0$&$0$\\ \hline
0&3&1,3&$1$&$1$\\ \hline
1&1&0&$0$&$0$\\ \hline
1&1&1&$\frac{1}{48}(b_1^2-3)$&$\frac{1}{24}(b_1-2)$\\ \hline
0&4&0&$\frac{1}{4}(b_1^2+b_2^2+b_3^2+b_4^2)$&$\frac{1}{2}(b_1+b_2+b_3+b_4)$\\ \hline
0&4&1,3&$0$&$0$\\ \hline
0&4&2&$\frac{1}{4}(b_1^2+b_2^2+b_3^2+b_4^2-2)$&$\frac{1}{2}(b_1+b_2+b_3+b_4-2)$\\ \hline
0&4&4&$\frac{1}{4}(b_1^2+b_2^2+b_3^2+b_4^2)$&$\frac{1}{2}(b_1+b_2+b_3+b_4-2)$\\ \hline
1&2&0&$\frac{1}{384}(b_1^2+b_2^2-8)(b_1^2+b_2^2)$&$\frac{1}{48}(b_1^2+b_2^2+b_1b_2-3(b_1+b_2))$\\ \hline
1&2&1&$0$&$0$\\ \hline
1&2&2&$\frac{1}{384}(b_1^2+b_2^2-6)(b_1^2+b_2^2-2)$&$\frac{1}{48}(b_1^2+b_2^2+b_1b_2-4(b_1+b_2)+5)$\\ \hline
1&3&0,2&$0$&$0$\\ \hline
1&3&1&$\frac{1}{4608}\big(\sum_{i=1}^3 b_i^6\hspace{-0.5mm}-\hspace{-0.5mm}20b_i^4\hspace{-0.5mm}+\hspace{-0.5mm}94b_i^2+6\sum_{i\neq j}b_i^2b_j^2(b_i^2-5)$

&$\frac{1}{96}\big(\sum_{i=1}^3 b_i^3\hspace{-0.5mm}-\hspace{-0.5mm}7b_i^2\hspace{-0.5mm}+\hspace{-0.5mm}14b_i+\sum_{i\neq j}b_ib_j(2b_i-5)$\\ 

&&&$+12b_1^2b_2^2b_3^2+3b_1^4-63b_1^2-15\big)$

&$+2b_1b_2b_3+b_1^2-5b_1-4\big)$\\ \hline

1&3&3&$\frac{1}{4608}\big(\sum_{i=1}^3 b_i^6\hspace{-0.5mm}-\hspace{-0.5mm}17b_i^4\hspace{-0.5mm}+\hspace{-0.5mm}103b_i^2+6\sum_{i\neq j}b_i^2b_j^2(b_i^2-5)$

&$\frac{1}{96}\big(\sum_{i=1}^3 b_i^3\hspace{-0.5mm}-\hspace{-0.5mm}8b_i^2\hspace{-0.5mm}+\hspace{-0.5mm}23b_i+2\sum_{i\neq j}b_ib_j(b_i-3)$\\ 

&&&$+12b_1^2b_2^2b_3^2-129\big)$

&$+2b_1b_2b_3-26\big)$\\ \hline
2&1&0&$0$&$0$\\ \hline
2&1&1&$\frac{1}{2^{16}3^35}(b_1^2-1)^2(5b_1^4-186b_1^2+1605)$&$\frac{1}{2^{9}3^25}(b-1)^2(b-4)(5b-22)$\\\hline
3&1&0&$0$&$0$\\ \hline
3&1&1&$\frac{1}{2^{25}3^65^27}(b^2\hspace{-1mm}-\hspace{-1mm}1)^2(b^2\hspace{-1mm}-\hspace{-1mm}3)^2(5b^6\hspace{-1mm}-\hspace{-1mm}649b^4\hspace{-1mm}+\hspace{-1mm}27995b^2\hspace{-1mm}-\hspace{-1mm}394695)$&$\frac{1}{2^{14}3^457}(b\hspace{-1mm}-\hspace{-1mm}1)^2(b\hspace{-1mm}-\hspace{-1mm}3)^2(b\hspace{-1mm}-\hspace{-1mm}6)(35b^2\hspace{-1mm}-\hspace{-1mm}462b\hspace{-1mm}+\hspace{-1mm}1528)$\\
\hline\hline
\end{tabular} 
\end{spacing}
\end{center}
%\end{table}

\normalsize

We can use theorem \ref{th:main}, the above table and the divisor equation (\ref{Mdivisor}) to compute the following expressions for stationary Gromov-Witten invariants of $\bp^1$.
\begin{itemize}
\item{Genus zero two-point invariants: \begin{align*}
\langle \tau_{2u_1}(\omega)\tau_{2u_2}(\omega)\rangle^{g=0} &=\frac{1}{u_1!^2u_2!^2} \frac{1}{(u_1+u_2+1)}\\
\langle \tau_{2u_1-1}(\omega)\tau_{2u_2-1}(\omega)\rangle^{g=0} &=\frac{u_1u_2}{u_1!^2u_2!^2}\frac{1}{(u_1+u_2)}
\end{align*}}
\item{Genus zero three-point invariants: \begin{align*}
\langle \tau_{2u_1}(\omega)\tau_{2u_2}(\omega)\tau_{2u_3}(\omega) \rangle^{g=0} &= \frac{1}{u_1!^2u_2!^2u_3!^2}\\
\langle \tau_{2u_1}(\omega)\tau_{2u_2-1}(\omega)\tau_{2u_3-1}(\omega) \rangle^{g=0}&= \frac{u_2u_3}{u_1!^2u_2!^2u_3!^2}
\end{align*}}
\item{Genus zero four-point invariants:
\begin{align*}
\langle \prod_{i=1}^4 \tau_{2u_i}(\omega)\rangle^{g=0}&=\frac{1}{\prod_{i=1}^4u_i!^2}(u_1+u_2+u_3+u_4+1)\\
\langle \prod_{i=1}^2 \tau_{2u_i}(\omega) \prod_{i=3}^4 \tau_{2u_i-1}(\omega)\rangle^{g=0}&=\frac{u_3u_4}{\prod_{i=1}^4u_i!^2}(u_1+u_2+u_3+u_4)\\
\langle \prod_{i=1}^4 \tau_{2u_i-1}(\omega)\rangle^{g=0}&=\prod_{i=1}^4\frac{u_i}{u_i!^2}(u_1+u_2+u_3+u_4)\\
\end{align*}}
\item{Repeatedly applying the divisor equation gives the even, genus zero $n$ point invariants:
\begin{align*}
\langle \prod_{i=1}^n \tau_{2u_i}(\omega)\rangle^{g=0}&=\frac{1}{\prod_{i=1}^nu_i!^2}(\sum_{i=1}^nu_i+1)^{n-3}\\
\end{align*}}
\item{Genus one one-point invariants: \begin{equation*}
\langle \tau_{2u}(\omega)\rangle^{g=1} =\frac{1}{24u!^2}(2u-1)
\end{equation*}}
\item{Genus one two-point invariants: 
\begin{align*}
\langle \tau_{2u_1}(\omega)\tau_{2u_2}(\omega)\rangle^{g=1} &=\frac{1}{24u_1!^2u_2!^2}(2u_1^2+2u_2^2+2u_1u_2-u_1-u_2)\\
\langle \tau_{2u_1-1}(\omega)\tau_{2u_2-1}(\omega)\rangle^{g=1} &=\frac{u_1u_2}{24u_1!^2u_2!^2}(2u_1^2+2u_2^2+2u_1u_2-3u_1-3u_2)
\end{align*}}
\item{Genus one three-point invariants:
\begin{align*}
\langle \tau_{2u_1}(\omega)\tau_{2u_2}(\omega)\tau_{2u_3}(\omega) \rangle^{g=1} &= \frac{1}{24\prod_{i=1}^3u_i!^2}\Big(\sum_{i=1}^3 2u_i^3\hspace{-1mm}-\hspace{-1mm}u_i^2+\sum_{i\neq j}u_iu_j(4u_i-1) +4u_1u_2u_3\Big)\\
\langle \tau_{2u_1}(\omega)\tau_{2u_2-1}(\omega)\tau_{2u_3-1}(\omega) \rangle^{g=1} &= \frac{u_2u_3}{24\prod_{i=1}^3u_i!^2}\Big(\sum_{i=1}^3 2u_i^3-5u_i^2+3u_i+\sum_{i\neq j}u_iu_j(4u_i-3) \\&+2u_1^2-3u_1-2u_2u_3+4u_1u_2u_3\Big)
\end{align*}}
\item{Genus two one-point invariants:
\begin{equation*}
\langle \tau_{2u}(\omega)\rangle^{g=2}=\frac{1}{2^{7}3^25u!^2}u^2(2u-3)(10u-17)
\end{equation*}}
\item{Genus three one-point invariants:
\begin{equation*}
\langle \tau_{2u}(\omega)\rangle^{g=3}=\frac{1}{2^{10}3^457u!^2}u^2(u-1)^2(2u-5)(140u^2-784u+1101).
\end{equation*}}
\end{itemize}


\begin{thebibliography}{99}


\bibitem{BEMSMat}
Borot, Ga\"etan; Eynard, Bertrand; Mulase Motohico and Safnuk, Brad.
\emph{A matrix model for simple Hurwitz numbers, and topological recursion.}
arXiv:0906.1206

\bibitem{BKMPRem} 
Bouchard, Vincent;  Klemm, Albrecht; Mari\~no, Marcos; Pasquetti, Sara 
\emph{Remodeling the B-Model.}
Comm. Math. Phys. {\bf 287}, 117-178 (2009).

\bibitem{BMaHur} Bouchard, Vincent and Mari\~no, Marcos 
\emph{Hurwitz numbers, matrix models and enumerative geometry.}
From Hodge theory to integrability and TQFT tt*-geometry, 263-283, Proc. Sympos. Pure Math., {\bf 78}, Amer. Math. Soc., Providence, RI, 2008.

\bibitem{DFMVol}
Dijkgraaf, Robbert ;  Fuji, Hiroyuki and Manabe Masahide 
\emph{The Volume Conjecture, Perturbative Knot Invariants,
and Recursion Relations for Topological Strings.}
arXiv:1010.4542

\bibitem{EynRec} Eynard, B. 
\emph{Recursion between Mumford volumes of moduli spaces}
arXiv:0706.4403

\bibitem{EMSLap} Eynard, Bertrand, Mulase Motohico and Safnuk, Brad.
\emph{The Laplace transform of the cut-and-join equation and the Bouchard-Marino conjecture on Hurwitz numbers.}
arXiv:0907.5224 

\bibitem{EynOrd} Eynard, Bertrand 
\emph{All orders asymptotic expansion of large partitions.}
arXiv:0804.0381.



\bibitem{EOrInv} Eynard, Bertrand and Orantin, Nicolas
\emph{Invariants of algebraic curves and topological expansion.}
Communications in Number Theory and Physics {\bf 1}
(2007), 347Ð452.

\bibitem{EOrTop} Eynard, Bertrand and Orantin, Nicolas
\emph{Topological recursion in enumerative geometry and
random matrices.}
J. Phys. A: Math. Theor. {\bf 42} (2009) 293001 (117pp).

\bibitem{FayThe}  Fay, John.
\emph{Theta functions on Riemann surfaces.} 
Springer Verlag, 1973.

\bibitem{GetTod}  Getzler, Ezra
\emph{The Toda conjecture.} 
Symplectic geometry and mirror symmetry (Seoul, 2000), 51-79, World Sci. Publ., River Edge, NJ, 2001. 

\bibitem{GetTop}  Getzler, Ezra
\emph{Topological recursion relations in genus 2} 
arxiv:math/9801003

\bibitem{MarOpe}  Mari\~no, Marcos 
\emph{Open string amplitudes and large order behavior in topological string theory.} 
J. High Energy Phys. {\bf 2008}, no. 3, 060, 34 pp.

\bibitem{OPaGro} Okounkov, A and Pandhariapande, R.
\emph{Gromov-Witten theory, Hurwitz theory, and completed cycles.}
Ann. of Math. (2) 163 (2006), no. 2, 517Ð560

\bibitem{OPaGro2} Okounkov, A and Pandharipande, R
\emph{Gromov-Witten theory, Hurwitz numbers, and Matrix models.}
Proc. Sympos. Pure Math., 80, Part 1, Amer. Math. Soc., Providence, RI, 2009. 

\bibitem{OPaVir} Okounkov, A and Pandharipande, R
\emph{Virasoro constraints for target curves}
Invent. Math. 163 (2006), no. 1, 47Ð108. 

\bibitem{NorStr}
Norbury, Paul
\emph{String and dilaton equations for counting lattice points in the moduli space of curves.}
To appear in {\em Trans.  AMS.}
arXiv:0905.4141 

\bibitem{NScPol} Norbury, Paul and Scott, Nick
\emph{Polynomials representing Eynard-Orantin invariants.}
arXiv:1001.0449 

\bibitem{Song} Song, Jun
\emph{Descendant Gromov-Witten invariants, simple Hurwitz number, and the Virasoro conjecture for $\bp^1$.}
Adv. Theor. Math. Phys. 3 (1999), no. 6, 1721Ð1768 (2000). 

\bibitem{Tue} Tuenter, Hans
\emph{Walking into an absolute sum}
Fibonacci Quart. 40 (2002), no. 2, 175Ð180. 

\bibitem{WitTwo} Witten, Edward
\emph{Two-dimensional gravity and intersection theory on moduli space.} Surveys in differential geometry (Cambridge, MA, 1990), 243--310, Lehigh Univ., Bethlehem, PA, 1991.



\end{thebibliography}
\end{document}